\documentclass[smallextended,referee,envcountsect,]{svjour3}
\usepackage[T5]{fontenc}
\usepackage[numbers,square]{natbib} % Ensures citations appear in square brackets
\usepackage{enumitem}
\smartqed
\usepackage{textgreek}
\usepackage{enumerate} 
\usepackage[T5]{fontenc}
\usepackage{amsfonts}
\usepackage{amssymb}
\usepackage[mathscr]{eucal}
\usepackage{amsfonts}
\usepackage{dsfont}
\usepackage[normalem]{ulem} % either use this (simple) or
\usepackage{soul} % use this (many fancier options)
\usepackage{nameref}    % Load the nameref package
\usepackage[bottom,stable]{footmisc}  % Ensures footnotes appear at the bottom

\usepackage{tikz}
\usetikzlibrary{decorations.pathreplacing,decorations.markings} 
\usetikzlibrary{shapes} 
\newcommand{\q}[1]{``#1''} %to create a quoting command 
\newcommand{\olsi}[1]{\,\overline{\!{#1}}} % overline short italic
\renewenvironment{case}[1]{
	\begin{trivlist}
		\item[\hskip \labelsep \textbf{Case #1:}] % Customize the label here
	}{
	\end{trivlist}
}
 
\usepackage{amsxtra,latexsym,amssymb, amscd}
\usepackage{pgfplots}
\usepackage{url}
\usepackage{color}

\usepackage{fancyhdr}
\usepackage{tikz}
\usepackage{graphicx}
\usepackage{subcaption}
\usepackage{caption}
\usepackage{epstopdf}
\usepackage{soul}
\usepackage{IEEEtrantools}
\usepackage[utf8]{inputenc} 
\usepackage{appendix} 
\usepackage{comment}
\usepackage[colorlinks,citecolor=blue,filecolor=black,linkcolor=blue,urlcolor=black]{hyperref}

\makeatother 
\numberwithin{equation}{section} % to give the equations numbers the form m.n (section-number.equation-number), use the \numberwithin command in the preamble of the document

\newtheorem{assumption}{Assumption} 

\newtheorem{lemmaalt}{Lemma}[theorem]
\newenvironment{lemmap}[1]{
	\renewcommand{\thelemmaalt}{#1}
	\lemmaalt
}{\endlemmaalt} 

\newtheorem{conditionalt}{Condition}[assumption] 
\newenvironment{conditionp}[1]{
	\renewcommand\theconditionalt{#1}
	\conditionalt
}{\endconditionalt} % New environment  New environment

\newtheorem{theoremalt}{Theorem}[theorem]
\newenvironment{theoremp}[1]{
	\renewcommand\thetheoremalt{#1}
	\theoremalt
}{\endtheoremalt}  % New environment  New environment

\newcommand{\rr}{\mathbb{R}}

\newcommand{\red}[1]{\textcolor{red}{#1}}

\usepackage{graphicx}
\journalname{JOTA}

\begin{document}

%\title{Instructions for Authors}
	\title{Hartman--Stampacchia theorems, Gale--Nikaid\^o--Debreu lemma, and Brouwer and Kakutani fixed-point theorems%\tnoteref{t1} 
	}
	%\tnotetext[t1]{The fourth author thanks Jean-Marc Bonnisseau for the comments in the very first version of the paper.}
	
%\subtitle{Using  the  LaTex Template}

\author{Pascal Gourdel, Cuong Le Van, Ngoc-Sang Pham, Cuong Tran Viet}

\institute{Pascal Gourdel\at
	Universit\'e de Paris 1 Panth\'eon Sorbonne \\ 
	Centre d'\'Economie de la Sorbonne (CES)\\
	pascal.gourdel@univ-paris1.fr
	\and
	Cuong Le Van \at
	CNRS, PSE, IPAG Business School, Hanoi Foreign Trade University\\
	cuong.levan@univ-paris1.fr
	\and 
	Ngoc-Sang Pham   \at
	EM Normandie Business School, M\'etis Lab\\
	npham@em-normandie.fr
	\and 
	Cuong Tran Viet  \at
	TIMAS - Thang Long University \\
	tranviet14@gmail.com           
}

\date{February 27, 2026}
%The correct dates will be entered by the editor.

\maketitle

\begin{abstract}
This paper uses the Hartman--Stampacchia theorems as the primary tool to prove the Gale--Nikaid\^o--Debreu lemmas. It also establishes a cycle of equivalences among the Hartman--Stampacchia  theorems, the Gale--Nikaid\^o--Debreu lemmas, and Kakutani and Brouwer fixed-point theorems. 
\end{abstract}

\vspace{0.2cm}

%\noindent Communicated by S\'andor Zolt\'an N\'emeth. 

%All acknowledgements should be placed in the back of the paper after Conclusions..
\section{ Introduction}\label{section:intro} 
%	\hspace{0.5cm}
Fixed point theorems are considered as an indispensable tool in various fields of mathematical economics. 
%The existence of an economic equilibrium was %simultaneously proven by \citet{ArrowDebreu1954} and \citet{McKenize1954}.
While the existence of an economic equilibrium could be established in many different ways, all classic proofs of the existence theorem rely on the Kakutani fixed--point argument (see Debreu's survey \cite{Debreu1982}). In the so-called \textit{excess demand} approach, the existence result is yielded from the existence of prices satisfying the \textit{Walras's law}. The core in this approach is a celebrated result known as the Gale--Nikaid\^o--Debreu lemma  \cite{Debreu1956,Debreu1959, Gale1955, Kuhn_1956, Nikaido56} (henceforth, GND lemma) %, 	%\citet{GaleMas-Colell1975, GaleMas-Colell1979}) 
whose proofs\footnote{We refer the interested readers to \cite{AliprantisBrownBurkinshaw1990,Debreu1982,Florenzano2003} for excellent treatments of the existence of equilibrium.} 
make use of the fixed--point theorems.\footnote{Debreu \cite{Debreu1956,Debreu1959}  and Nikaid\^o  \cite{Nikaido56} used the Kakutani fixed point theorem, Gale \cite{Gale1955} used the Knaster--Kuratowski--Mazurkiewicz lemma, and Kuhn \cite{Kuhn_1956}  made use of the Eilenberg and Montgomery fixed point theorem. 
In addition, Khan, McLean and Uyanik  \cite{KhanMcLeanUyanik_ETB25} recently provide  alternative proofs of the GND and Knaster--Kuratowski--Mazurkiewicz lemmas by using qualitative and generalized games.

The GND lemma has also been extended to more general settings. For instance, Yannelis \cite{YannelisJMAA1985}, Krasa and Yannelis \cite{KrasaYannelisET1994}, Cornet, Gou and Yannelis \cite{CornetGouYannelis2023} proved other versions of the GND lemma when the correspondence is upper demi-continuous (u.d.c.) (note that an upper semi--continuous correspondence is also u.d.c). For extensions of the GND lemma to infinite-dimensional frameworks, see \cite{CornetETB2020,CornetGouYannelis2023,Otsuka_ETB2024,YannelisJMAA1985}. For extensions of the GND lemma to settings with discontinuous preferences, see \cite{CornetETB2020,HeYannelis_JMAA2017,KhanMcLeanUyanik_ET2025,YannelisJMAA1985}.}

%Reviewer #2: The paper seems like a solid piece of work. However, no credit is given to earlier work on the GND Lemma. For example the weak Walras law was used in Yannelis JMAA 85, and Krasa-Yannelis ET 1994, together with the upped demi continuity of the excess demand correspondence (weaker than use). Also there are related papers that relax the continuity of the excess demand, e.g., He-Yannelis JMAA and ET 2017/2018 , Cornet ETB 2020 and others...The recent paper of Khan et all in ET 2025  is not mentioned. I am glad to accept a revised version of this paper but credit should be given to earlier contributions on the subject. The references should be updated and perhaps add open questions. It is not clear how this work could be extended to infinite dimensional spaces and also allow for discontinuous exceed demands.

As mentioned by Duppe and Weintraub \cite{DuppeWeintraub2014} and Khan \cite{Khan2021}, 
Debreu sought to explore the possibility of proving the GND lemma and its generalization without relying on the fixed-point argument. %(see \cite{DuppeWeintraub2014} and \cite{Khan2021})
Le et al. \cite{Le2022Sperner}
%The authors  
have been among the first to investigate this question by providing a proof of the GND lemma using Sperner's lemma (a combinatorial result on colorings of triangulations). %as the primary argument to prove the GND lemma. 
However, these authors %\cite{Le2022Sperner} 
do not prove the generalized versions of the GND lemma \cite{Florenzano2003,Florenzano1982}. 
%, given by  \cite{Florenzano1982} (or \cite{Florenzano2003}).%, by the means of Sperner's lemma.

Our paper aims to address Debreu's question by  providing a new proof of the GND lemma and its generalized versions derived directly from the Hartman--Stampacchia theorem\footnote{See Lemma $3.1$ in \cite{HartmanStampacchia1966}.} \cite{HartmanStampacchia1966}, a well-known result in the theory of variational inequalities.\footnote{D’Agata \cite{dAgata_MSS2022} extends the Hartman--Stampacchia theorem (\autoref{HS-thm}) to unbounded cases and applies it to establish the existence of a general equilibrium when the excess demand correspondence is single-valued.}

To achieve our goal, a generalized version of the Hartman--Stampacchia theorem (\autoref{HS-thm-generalized}) associated with upper semi-continuous and convex-valued correspondences is provided and proven.\footnote{Note that we also obtain a generalized version of Hartman--Stampacchia's theorem for the case of lower semi-continuous correspondence (see  \autoref{HS-thm-generalized-lwr} below).} Since we work with correspondences, we make use of some additional tools such as finite covering of a compact set, a partition of unity subordinated to a covering and the Carath\'eodory convexity theorem \cite{Caratheodory1907}. 	 This theorem %, %together with the original Hartman-Stampacchia theorem (\autoref{HS-thm}),
allows us to prove not only the GND lemma but also its  generalized version.  %\cite{Florenzano2003,Florenzano1982}. 
% in \cite{Florenzano1982} or \cite{Florenzano2003}. 

The second part of our paper establishes an equivalence cycle among the Hartman--Stampacchia theorems, the GND lemmas and some related fixed-point theorems. %Firstly, we use Hartman-Stampacchia's theorem to prove the Brouwer fixed-point theorem. Notice that \cite{KinderlehrerStampacchia2000} uses the Brouwer fixed-point theorem.
The original Hartman--Stampacchia theorem for continuous mappings (\autoref{HS-thm}, HS1) implies the one for upper-semi continuous correspondences 
(\autoref{HS-thm-generalized}, HS2), which yields, in turn, the GND lemma (\autoref{Florenzano1982-thm}). As shown in \cite{Florenzano1982}, the GND lemma implies the Kakutani theorem, which straightforwardly entails the Brouwer theorem. Finally, the original Hartman-Stampacchia theorem for continuous mappings can be derived from the Brouwer theorem, thus completing a cycle of equivalences. Moreover, in Section \ref{Equivalnce-relation}, we prove the Kakutani and Brouwer theorems using the original Hartman--Stampacchia theorem.  %(\autoref{HS-thm}, HS1). 
All of these results are illustrated in \autoref{fig-equivalence}.  
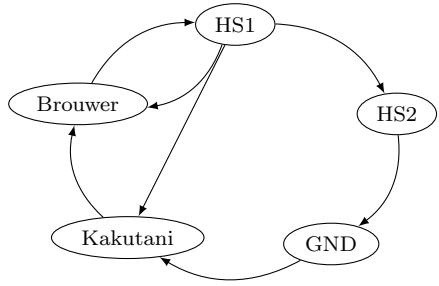
\begin{figure}[h!]
	\centering
	%\documentclass{article}
	%\usepackage{tikz} 
	%\usetikzlibrary{shapes}
	%\begin{document} 
	%	\begin{figure} 
		\begin{tikzpicture} [scale = 0.20] 
			
			% define the nodes
			\node[ellipse, draw] at (11*cos 90, 8*sin 90) (a) {HS1};
			\node[ellipse, draw] at (11*cos 15, 8*sin 15) (b) {HS2}; 
			\node[ellipse, draw] at (11*cos 305, 8*sin 305) (c) {GND};
			\node[ellipse, draw] at (11*cos 230, 8*sin 230) (d) {Kakutani};
			\node[ellipse, draw] at (11*cos 160, 8*sin 160) (e) {Brouwer};
			% draw the paths 
			%            \draw[->] (a) to [bend left=20] (b);
			\draw[-latex, bend left] (a) edge (b); 
			%			\draw[->] (b) to [bend left=20] (c); 
			\draw[-latex, bend left] (b) edge (c);
			%			\draw[->] (c) to [bend left=20] (d); 
			\draw[-latex, bend left] (c) edge (d);
			%			\draw[->] (d) to [bend left=20] (e); 
			\draw[-latex, bend left] (d) edge (e);
			\draw[-latex, bend left] (e) edge  (a); 
			\draw[-latex, bend left] (a) edge (e);
			\draw[-latex ] (a) edge  (d);
		\end{tikzpicture}  
		%	\end{figure}
	%\end{document}
	%	\input{Picture/equivalence1.tex}
	%	\input{Picture/equivalence.tex}
	\caption{\footnotesize {The cycle of equivalences: Hartman--Stampacchia theorem (\autoref{HS-thm}, HS1), generalized version of Hartman--Stampacchia for correspondence (\hyperref[HS-thm-generalized]{Theorem \ref{HS-thm-generalized}}, HS2), GND lemma, Kakutani and Brouwer fixed-point theorems}.}
	\label{fig-equivalence}
\end{figure}

Our paper is related to the connection between the theory of variational inequalities and general equilibrium theory. Using variational inequalities,  \cite{mor07,zhao93} investigate the existence of general equilibrium in economies with production while \cite{donato18,donato23}
prove the existence of general equilibrium in economies with incomplete financial assets. Our paper contributes to this literature because the GND lemma is  key for equilibrium existence.
%	Some authors \citep{zhao93, mor07,  donato18, donato23} use variational inequalities to prove the existence of general equilibrium existence in different economic models. 
%	\cite{zhao93}, \cite{donato18}, \cite{donato23}, \cite{zhao93}, \cite{mor07}.

There have been many efforts to provide a proof of the existence result without using Kakutani's fixed-point argument, including studies by \cite{BarbollaCorchon1989, Fraysse2009, Greenberg1977, John1999, Mackowiak2010, Quah2008}. These studies generally fall into two categories, both of which impose additional conditions on \textit{excess demand function} or \textit{correspondence}. The first category relies on the \textit{gross substitutes} assumption (e.g., \cite{BarbollaCorchon1989, Fraysse2009, Greenberg1977}) while the second is based on the  
\textit{weak axiom of revealed preference} assumption (e.g., \cite{Mackowiak2010, Quah2008}). Since we provide a new proof of the GND lemma without using the Kakutani fixed-point theorem but instead using the the Hartman--Stampacchia theorem, our paper contributes to this line of research.

The paper proceeds as follows. In the next section, we begin with some notations and definitions, the Hartman--Stampacchia theorem (\autoref{HS-thm}) and its generalization 
(\autoref{HS-thm-generalized}), which are of fundamental importance for the proof of the GND lemma and its extended version. Section \ref{proofs} contains the main result of the
paper, namely the proofs of the GND lemma and its extended version without using the Kakutani fixed-point argument. Section \ref{Equivalnce-relation} is dedicated to showing how the Hartman--Stampacchia theorem can be used to prove the Brouwer and Kakutani theorems, establishing the cycle of equivalences. %Additionally, we demonstrate in Section \ref{Hartman-Stampacchia} that Hartman-Stampacchia theorem is a consequence of Brouwer theorem. 
Finally, some extended and alternative proofs are given in  %\ref{proofs-sec} 
\autoref{proofs-sec}.

\section{ Preliminaries}\label{intro:sec} 
\subsection{Notations and Definitions}
We start by introducing some notations that we use in this paper.  We denote the following:
\begin{itemize}[itemsep=0.7em, label=$\bullet$]
	\item $\langle x, y \rangle$ as the inner product and $||x||$ as the Euclidean norm for $x, y \in \mathds{R}^N$,
	\item 
$0_N$ and $0_m$ as the zero vectors in $\mathds{R}^N$ and $\mathds{R}^m$,
%\item $B(x, r)$ and $\olsi B(x, r)$ as the open and closed balls in $\mathds{R}^N$ centered at $x$ (or $0_N$) with radius $r$ (or $1$), and $S$ as the unit sphere of $\olsi B$.

\item $B(x,r)$ and $\olsi B(x,r)$ (resp. $B$ and $\olsi B$)  as the open and closed balls in $\mathds{R}^N$ centered at $x$ (resp $0_N$) and radius $r$ (resp $1$) and $S$ as the unit-sphere associated with $\olsi  B$, 
\item 
$P^\circ = \{ z \in \mathds{R}^N \mid \langle p, z \rangle \leq 0 \text{ } \forall p \in P \}$ as the polar cone of $P$,
\item 
$\Delta = \{ x \in \mathds{R}^N \mid \sum_{i=1}^{N} x_i = 1, x_i \geq 0 \text{ for } i=1,2,...,N \}$ as the unit simplex,
\item 
$N_C(x) = \{ u \in \mathds{R}^N \mid \langle u, y - x \rangle \leq 0 \text{ } \forall y \in C \}$ as the normal cone to $C$ at $x$,
\item 
$P^\mathsf{c}$ as the complement of $P$ for $P \subset \mathds{R}^N$,
\item 
$span(P) = \{ \sum_{i=1}^{k} \beta_i u_i \mid k \in \mathds{N}, \beta_i \in \mathds{R}, u_i \in P \}$ as the span of $P$,
\item 
$\text{int}(P) = \{ x \in P \mid \exists r > 0, B(x, r) \subset P \}$ as the interior of $P$,
\item $\text{ri}(P)$ as the relative interior of $P$, i.e., $\text{ri}(P)\equiv \{x\in P  \mid \exists r > 0, B(x, r)\cap \text{aff}(P) \subset P \}$, where aff$(P)$ is the affine hull of $P$,\footnote{Recall that aff$(P)$ is the intersection of all affine sets of $\rr^N$ containing $P$.}
\item 
$\text{Bd}^\text{r}(P) = \olsi P \setminus \text{ri}(P)$ as the relative boundary of $P$, where $\olsi P$ denotes the closure of $P$.
\end{itemize}
\hspace{0.5cm}
Let us also introduce some basic concepts. A subset $P \subset \mathds{R}^N$ is called a \textit{cone with vertex $0_N$} if, for every $x\in P$ and $\lambda > 0$, we have $\lambda x \in P$.  A \textit{convex cone} is a cone that is a convex set.

%recall the definition and some properties of \textit{upper semi-continuous correspondence}.

Let $X,Y$ be non-empty topological spaces. 
A correspondence $\Gamma: X \rightrightarrows Y$ is \textit{upper semi-continuous (u.s.c.)} at a point $ x \in X$
if, for every open set $V \subset Y$ such that $\Gamma(x)\subset V$, there exists a neighborhood $U \subset X$ of $x$ such that $\Gamma(U)\subset V$. The correspondence  $\Gamma$ is \textit{upper semi-continuous on $X$} if it is upper semi-continuous at every point in $X$.  

The correspondence $\Gamma$ is \textit{lower semi-continuous (l.s.c.)} at a point $x \in X$ if, for every open set $V\subset Y$ such that $\Gamma(x)\cap V\not=\emptyset$, there exists a neighborhood $U \subset X$ of $x$ such that for all $z\in U$, we have $\Gamma(z)\cap V\not=\emptyset$. The correspondence  $\Gamma$ is \textit{lower semi-continuous} on $X$ if it is lower semi-continuous at every point in $X$.

%Recall that if $\Gamma$ is single-valued, the notions of \textit{continuity}, \textit{upper semi-continuity}, and the \textit{lower semi-continuity} turn out to be equivalent.  

%(i) $\Gamma (x) $ is compact, non-empty, and (ii) for any sequence $\{x_n\}$ converging to $x$, for any sequence $\{y_n\}$  with $y_n \in \Gamma (x_n), \forall n$, there exists a subsequence $\{y _{n_k}\}$ which converges to $y\in \Gamma (x)$.
%\end{definition}

Notice that if $Y$ is compact, then $\Gamma$ is upper semi-continuous if and only if $\Gamma$ is closed, that is, the graph of $\Gamma$ is closed. Additionally, if $\Gamma$ is upper semi-continuous and $K\subset X$ is compact, then (1) $\Gamma (K) $ is compact if Y is compact and  $\Gamma$ has closed values, and (2) $\Gamma(K)$ is compact if $\Gamma$ is compact valued.

\subsection{The Hartman--Stampachia Theorem and its Generalization} \label{HS:sec}
Let us recall the following result from \citep{HartmanStampacchia1966} (see Lemma $3.1$ on page $276$). %in the finite dimensional setting. 
%We state with more details the Hartman-Stampacchia theorem. 
\begin{theorem}[Hartman--Stampacchia theorem]  \label{HS-thm}
	Let $K$ be a compact convex set in $\mathds R^N$, $f$ a continuous mapping from $K$ to $\mathds R^N$.
	%\begin{enumerate} [(i)]
	%\item 
	Then, there exists $\olsi  u \in K$ such that 
	\begin{equation} \label{HS inequality}
		\langle v, f(\olsi  u) \rangle \leq \langle \olsi  u, f(\olsi  u) \rangle \quad \forall v \in K. 
	\end{equation} 
	%\item Assume that $\partial K$ is $C^1$. If the inequality \eqref{HS inequality} fails for all $u_0 \in \partial K$, there exists $\olsi {x} \in C$ such that $f(\olsi {x}) =0$. 
	%\end{enumerate}
\end{theorem} 
%\begin{remark} 
%		In \cite{KinderlehrerStampacchia2000} (Theorem I.3.i), 
The Hartman--Stampacchia theorem (\autoref{HS-thm})
was proven by Hartman and Stampacchia \cite{HartmanStampacchia1966} 
 using the index theory. Notably, the point $f(\olsi {u}) \in N_K(\olsi {u})$, meaning $f(\olsi u)$ lies the normal cone $N_K(\olsi u)$ to the set $K$ at $\bar u$. 
%	\hyperref[HS-thm]{Theorem \ref{HS-thm}} can be proven by using  the index theory or the topological degree theory.  
%	The proof could be found in \cite{KinderlehrerStampacchia2000} 
%	(TO PROVIDE Reference). 
%\end{remark}
Before proceeding with a detailed generalization of \autoref{HS-thm}, we introduce a useful lemma 
%\footnote{A very similar idea in \autoref{representation-lemma} can be found in \cite{Cellina1969a}. However, it is used for different purposes. While  in Theorem $1$'s \cite{Cellina1969a} it is used to build a continuous mapping whose graph is separated with the correspondence with a small distance, in our paper it is used to extend the Hartman Stamppachia's theorem to correspondence.}
%\footnote{ A very similar idea in \autoref{representation-lemma} can be found in \cite{Cellina1969a}. However, the formulation of the mapping $f$ in the lemma differs from that of Theorem 1 in \cite{Cellina1969a} since both mappings are based on $2$ different structures of finite covering. Besides, 	this version was written before reading Cellina's paper.and is simpler due to the finite dimensional setting. }           
describing the value of a continuous mapping in terms of a finite linear combination of vectors. We will later see that \autoref{representation-lemma} below plays a critical role in the proof of the generalized Hartman--Stampachia  theorem.% and enables us to use the \q{compactness argument}.
%Note that the proof of this lemma involves \autoref{representation-lemma} stated in \autoref{proofs}. 
\begin{lemma}\label{representation-lemma}
	Let $C$ be a non-empty and compact set in $\mathds{R}^N$, $\zeta$ a non-empty valued correspondence from $C$ to $\mathds R^N$. 
	Let $r>0$. Then, there exists a continuous mapping $f: C \to \mathds{R}^N$ satisfying the following condition: 
	\begin{conditionp}{R} 
		\label{representation00}
		For each $x \in C$, there exist at most $N+1$ vectors $z^1,\dots,z^{N+1}$ in $\zeta \big(B(x,r)\big)$ and positive numbers $\beta_1,\dots, \beta_{N+1} $ such that\footnote{Note the convention that superscripts are used for labelling vectors while subscripts denote real numbers. For example, as in \hyperref[representation00]{Condition \ref{representation00}}, the parameters $\beta_1,\dots, \beta_{N+1}$ are real numbers, and $z^1,\dots, z^{N+1}$ are vectors belonging to the finite-dimensional space $\mathds R^N$.}
		\begin{equation}  \label{representation-condition}
			f(x)= \sum_{i=1}^{N+1} \beta_i z^i	    
		\end{equation}
		with $\sum_{i= 1}^{N+1} \beta_i =1$.
	\end{conditionp} 
\end{lemma}
\noindent %\hyperref[prove-representation-lemma]{\nameref{prove-representation-lemma}} 
See the %\hyperref[prove-representation-lemma]{\nameref{prove-representation-lemma}} 
\nameref{prove-representation-lemma} 
in 	\autoref{proofs-sec}. %\hyperref[proofs-sec]{Proofs Section} %on \autopageref{prove-representation-lemma}.  

\noindent Note that the $N+1$ vectors $z^1,\dots,z^{N+1}$ and the numbers $\beta_1, \dots,\beta_{N+1}$ are allowed to depend on the parameters $x,r$ and $M$ functions $(\alpha_i)_{i=1}^M$ (see these functions in the proof). However, for simplicity, these dependencies are omitted. 

\begin{remark} 
	\autoref{representation-lemma} 
can be interpreted as follows: the value $f(x)$ of a continuous mapping is expressed as a convex combination of at most $N+1$ elements from $\zeta\big( B(x,r)\big)$. This structure, where the number of terms in the combination is fixed, enables us to apply the \q{\textit{compactness argument}}. 
\end{remark}
\begin{remark} \label{representaion-extra}
	In 	\autoref{representation-lemma}, if we additionally assume that $C$ is convex in $\mathds{R}^N$ and the correspondence $\zeta$ is from $C$ to  $C$, then  the conclusion regarding the existence of a continuous mapping $f$ satisfying \hyperref[representation00]{Condition \ref{representation00}}  remains the same, but such a mapping $f$ is defined from $C$ into $C$. 
	%		The conclusion of \autoref{representation-lemma} still holds under  two extra assumptions. More precisely, in addition to the assumptions of \autoref{representation-lemma}, assume that $C$ is convex in $\mathds{R}^N$ and the correspondence $\zeta$ is from $C$ into $C$. The conclusion of the existence of a continuous mapping $f$ satisfying \hyperref[representation00]{Condition \ref{representation00}}  is the same. Besides, the mapping $f$ is admitted from $C$ into $C$. 
\end{remark}   

We now introduce an extension of  
 \hyperref[HS-thm]{Hartman--Stampacchi theorem} 
(\autoref{HS-thm}) to correspondences. This extension concerns some characteristics of the correspondence. The extension of Hartman--Stampacchia theorem to the case of upper semi-continuous correspondence is a key step in proving the generalized GND lemma. The precise formulation of the generalization is stated in \autoref{HS-thm-generalized}. 
%One of appropriate generalization of \nameref{HS-thm} is to replace the function $f$ by a correspondence $\zeta$. With this idea in mind, the vector $\olsi  u$ stated in \autoref{HS-thm} is kept the same, but the vector $f(\olsi  u)$ is substituted by a vector $z$ in $\zeta(\olsi  u)$. We state the theorem below. 

The proof of the theorem builds on the original  \autoref{HS-thm} and relies on the \q{\textit{compactness argument}} (\autoref{representation-lemma}), which involves the concepts of \textit{unity subordinated to a covering} and the \textit{Carath\'eodory convexity theorem}.\footnote{Browder \cite{Browder68} (Theorem $6$) provides a generalized version of \autoref{HS-thm-generalized} for a locally convex space. While we work in finite-dimensional spaces, our proof appears to be simpler than his.} Additionally, for completeness, we include another generalized Hartman--Stampacchia theorem for \textit{lower semi-continuous correspondences}, expressed in \autoref{HS-thm-generalized-lwr}. The proof of this corollary makes use of a continuous selection theorem based on the work of Michael \cite{Michael1956}. 
\begin{theorem} [Hartman--Stampacchia theorem for upper semi-continuous correspondence] \label{HS-thm-generalized} Let $C$ be a compact convex set in $\mathds R^N$, $\zeta$ a non-empty, convex and compact valued  correspondence from $C$ to $\mathds R^N$. If $\zeta$ is upper semi-continuous, then there exist some $x \in C$ and $z \in \zeta(x)$ such that 
	\begin{equation*}
		\langle p,z \rangle \le \langle x,z \rangle \quad  \forall p \in C.  
	\end{equation*} 
\end{theorem}  
\noindent See \nameref{proof-HS-thm-generalized} in the 
\autoref{proofs-sec}.
%\hyperref[proofs-sec]{Proofs Section} 
%on \autopageref{proof-HS-thm-generalized}.  
%\subsection{Proof of \autoref{HS-thm-generalized}}  

%The proof of it is postponed in the appendix. 
%\noindent See \nameref{proof-HS-thm-generalized} in Appendices on \autopageref{proof-HS-thm-generalized}. 

%\begin{remark}	\end{remark}

\begin{corollary} [Hartman--Stampacchia theorem for lower semi - continuous correspondence] \label{HS-thm-generalized-lwr} Let $C$ be a compact convex set in $\mathds R^N$, $\zeta$ a non-empty, convex and compact valued  correspondence from $C$ to $\mathds R^N$. If $\zeta$ is lower semi-continuous, then there are some $x \in C$ and $z \in \zeta(x)$ such that 
	\begin{equation*}
		\langle p,z \rangle \le \langle x,z \rangle \quad  \forall p \in C.  
	\end{equation*} 
\end{corollary} 
\noindent See \nameref{proof-HS-thm-generalized-lwr} in \autoref{proofs-sec}.% on \autopageref{proof-HS-thm-generalized-lwr}. 
\iffalse 
\begin{remark} From identity \eqref{convergence}, if we define a mapping $f:X\mapsto \mathds R^N$ by $f(x) = \lim\limits_{k \to \infty} f^{n_k}(x_{n_k})$, we can see $f$ as a selection mapping of $\zeta$; it is possible that the selection mapping $f$ is not continuous on $X$. This is, of course, in marked contrast to the property of the selection mapping in the proof for lower semi-continuity, namely \autoref{selection-thm}, where the selection mapping is continuous. 
	%	Comparing this with the lower continuity case, namely \autoref{HS-thm-generalized-lwr}, the selection mapping is continuous. 
	% 	 \red{Should be a small remark about the proof using the mapping $f$}
\end{remark} 
\fi 
\section{Gale--Nikaid\^o--Debreu Lemmas }\label{proofs} 
In this section, we aim 
to provide proofs for not only the
%the Gale-Nikkadô-Debreu lemma 
\hyperref[GNDstrong]{Gale--Nikaid\^o--Debreu lemma} (\autoref{GNDstrong})
but also its generalized version (\autoref{Florenzano1982-thm}). The proofs are constructed using the \hyperref[HS-thm]{Hartman--Stampacchia} theorem and the \hyperref[HS-thm-generalized]{extended  Hartman--Stampacchia} theorem, with detailed formulations of theses results given in  Section \ref{GND-lemma}. The main arguments involved in the proofs are \autoref{representation-lemma}, \autoref{HS-thm-generalized},
 and the concept of retract mapping. The first of these is introduced in Section \ref{HS:sec}, while the notion of retract mapping, along with a supporting lemma, is discussed in Section \ref{rtrct-mppng}.  Finally, the direct proofs for % \hyperref[GNDstrong]{Theorems}
Theorems \ref{GNDstrong} and \ref{Florenzano1982-thm} are provided in Sections \ref{prv-GND} and \ref{Flo82-prf}, respectively. 
%	Finally, 
%	by the means of \nameref{HS-thm} (\autoref{HS-thm}) and its generalized version (\autoref{HS-thm-generalized}), we give the direct proofs in Sections \ref{prv-GND} and \ref{Flo82-prf} respectively.

%The proofs of \autoref{GNDstrong} and \autoref{Florenzano1982-thm} are given in Sections \ref{prv-GND} and \ref{Flo82-prf} respectively. 
\subsection{Gale--Nikaid{\fontencoding{T5}\selectfont \ocircumflex}--Debreu Lemma and its Generalized Version} \label{GND-lemma}
% \nameref{GNDstrong} (\autoref{GNDstrong}) and its generalized version by Florenzano(1982) ( \autoref{Florenzano1982-thm}) by using  Hartman-Stampacchia theorem and its generalized version. For this purpose, we deal with two separate cases: the correspondence $\zeta$ is either single-valued or multi-valued. In the latter case, $\zeta$ is an upper semi-continuous correspondence with convex, compact values; the proof of this case splits into two subcases.  
%The existence of an economic equilibrium was proven by 
Arrow and Debreu 
\cite{ArrowDebreu1954} and McKenzie  \cite{McKenize1954} simultaneously proved a fundamental equilibrium existence result in theoretical economics. 
%Arrow and Debreu (1954) 
%use the Kakutani's fixed-point theorem as a tool to prove the existence of an economic equilibrium. 
Later, with the papers  of \cite{Debreu1956, Gale1955, Nikaido56, Kuhn_1956}, a celebrated formulation known as  
\hyperref[GNDstrong]{Gale--Nikaid\^o--Debreu lemma} (\autoref{GNDstrong})
%Gale-Nikaid\^o-Debreu lemma 
 or the excess demand theorem \cite{CornetGouYannelis2023} was developed, providing a deeper economic explanation. 
%Later Flonrenzano(1982) gives a generalization of this lemma.  
%The original Gale-Nikaido-Debreu lemma is given in 
%The core argument for the classic proofs\footnote{We refer to \cite{Debreu1959,Debreu1982} for more details.} of the lemma is based on the Kakutani fixed-point theorem. 
Let us recall the GND lemma. 
% The proof makes use of Kakutani Theorem.    
\begin{theorem}[Gale--Nikaid\^o--Debreu lemma]\label{GNDstrong}
	Let $\Delta$ be the unit-simplex in $\mathds R^N$.  Let $\zeta$ be an upper semi-continuous correspondence with non-empty, compact, convex values from $\Delta$ to $\mathds R^N$. Suppose $\zeta$ satisfies the following condition:
	\begin{align}\label{gnd-condition}
		\forall p \in \Delta , \;\; \forall  z\in \zeta (p), \;\; \langle p, z \rangle  \leq 0.
	\end{align}
	Then, there exists $\olsi  p \in \Delta$ such that $\zeta (\olsi  p )\cap\mathds R^N _- \neq \emptyset $.
\end{theorem}  

\begin{remark} % \red{The fact that the price set considered is $\Delta$, eliminated the case where the price might be negative (see \autoref{ngtv-prc}).} Besides, 
	An equivalent statement of %Theorem 
	\autoref{GNDstrong} 
	%\ref{GNDstrong} 
	is obtained by replacing the strong condition \eqref{gnd-condition} by  the weak condition \eqref{gnd-condition2}  below
	\begin{equation} \label{gnd-condition2} 
		\forall  p \in \Delta,\; \exists z \in \zeta(p)  \text{ such that } \langle p, z \rangle  \le 0. 
	\end{equation} 
	It is clear that Condition \eqref{gnd-condition} implies Condition \eqref{gnd-condition2}. Conversely, assume that the correspondence $\zeta$ satisfies Condition \eqref{gnd-condition2}. We define the correspondence $\zeta^\prime: \Delta \to \mathds R^N$ by  $\zeta^\prime(p) = \{ z \in \zeta(p): \langle p, z \rangle  \le 0 \}$. It follows that $\zeta^\prime$ is non-empty, convex, compact valued and upper semi-continuous correspondence from $\Delta$ to $\mathds{R}^N$ such that 
	\begin{align*}%\label{gnd-condition}
		\forall p \in \Delta , \; \forall  z\in \zeta^\prime (p),  \;\; \langle p, z \rangle  \leq 0.
	\end{align*}
	From \autoref{GNDstrong}, there exits $\olsi {p} \in \Delta$ such that $\zeta^\prime(\olsi {p}) \cap\mathds{R}^N _- \neq \emptyset $. Since $\zeta^\prime(\olsi {p}) \subset \zeta(\olsi {p})$, it follows  $\zeta (\olsi  p )\cap\mathds R^N _- \neq \emptyset $. 
\end{remark}  
\noindent If the correspondence $\zeta$ is single-valued, \autoref{GNDstrong} is restated as follows:  
\begin{theoremp}{\ref{GNDstrong}$^\prime$} \label{key}
	Let $\Delta$ be the unit-simplex in $\mathds R^N$. Let $\zeta$ be a continuous mapping from $\Delta$ to $\mathds R^N$. Suppose $\zeta$ satisfies the following condition:
	\begin{align*}\label{gnd-condition-mapping}
		\forall p \in \Delta ,\; \langle p, \zeta(p) \rangle  \leq 0.
	\end{align*}
	Then, there exists $\olsi  p \in \Delta$ such that $\zeta (\olsi  p ) \in \mathds R^N _- $.
\end{theoremp}
We now turn our attention to a generalization of \autoref{GNDstrong},  which was established by Florenzano  \cite{Florenzano1982}. One of the contributions of \cite{Florenzano1982} is to provide the proof of a version of the GND lemma (\autoref{GNDstrong})  without using the Kakutani fixed-point argument. As a result, 
the Kakutani fixed-point theorem is implied as a direct consequence. She proves the GND lemma by contradiction,\footnote{See Lemma $1$ in \cite{Florenzano1982} on page $115$ or Lemma $2.1.1$ in \cite{Florenzano2003} on page $45$.} applying the strict separation theorem for two disjoint convex sets $-$ one of which is closed and the other compact. Additionally, she applies a partition of unity subordinated to a covering, together with the Brouwer fixed-point theorem. 

\begin{theorem}[Lemma $1$, \cite{Florenzano1982}] \label{Florenzano1982-thm}
	Let $P$ be a closed convex cone with vertex $0_N$ in $\mathds{R}^N$. Let  $\zeta$ be an upper semi-continuous and non-empty, compact convex valued correspondence from $\olsi  B \cap P$ to $\mathds{R}^N$. If $\zeta$ satisfies the following condition 
	\begin{equation}  \label{Florenzano1982-condition}
		\forall p \in S \cap P, \;\; \exists z \in \zeta(p) \text{ such that }  \langle p, z \rangle  \le 0,  
	\end{equation} 
	then there exists $\olsi {p} \in \olsi  B \cap P$ such that  $\zeta(\olsi {p}) \cap P^\circ \ne \emptyset$. 
\end{theorem} 

\begin{remark} 
	Obviously, without loss of generality, we can replace condition \eqref{Florenzano1982-condition} by  condition \eqref{Florenzano1982-condition2} below
	\begin{equation} \label{Florenzano1982-condition2} 
		\forall p \in S \cap P, \;\; \forall z \in \zeta(p) \text{ such that } \langle p, z \rangle  \le 0.
	\end{equation} 
\end{remark} 

\begin{remark}
	On one hand, \autoref{Florenzano1982-thm} can be viewed as a generalized version of the GND lemma by allowing the convex cone $P$ to take various forms such as  the entire space $\mathds R^N$,  the positive orthant of $\mathds R^N$ or the closed half space. As a result, \autoref{Florenzano1982-thm} encompasses a range of important cases in the literature. 
	%there are various formulation of the GND established. 
	On the other hand, \autoref{Florenzano1982-thm} does not exclude the non-zero prices.  In \cite{FlorenzanoLeVan1986}, the following example is provided, demonstrating that, in general, the vector $\olsi  p$ in  \autoref{Florenzano1982-thm} might be $0_N$.
\end{remark}  
%	It is also worth noting that in general under the conditions of \autoref{Florenzano1982-thm}, it is possible that $\olsi  p$ might be $0_N$. This is made clear in the following example.
\begin{example} 
	Consider a cone $P=\mathds R^2$  and a single-valued correspondence $\zeta$ from $\olsi B$ into $\mathds R^2$ defined as follows: $\zeta(p)= -p$ for all $p \in \olsi B$. Obviously, all conditions of \autoref{Florenzano1982-thm} hold. Indeed, it is easy to see that  $\zeta(\olsi  p) \cap P^\circ \ne \emptyset$ if and only if $\olsi  p = 0_2$.
\end{example}	

\noindent In the case $\zeta$ is single-valued, \autoref{Florenzano1982-thm} is rewritten as follows: 
\begin{theoremp}{\ref{Florenzano1982-thm}$^\prime$} 
	\label{Florenzano1982-thm-mapping}
	Let $P$ be a closed convex cone with vertex $0_N$ in $\mathds{R}^N$. Let $\zeta$ be a continuous mapping from $\olsi  B \cap P$ to $\mathds{R}^N$. If $\zeta$ satisfies the condition 
	\begin{equation} 
		\label{Florenzano1982-condition-mapping} 
		\forall p \in S \cap P, \;\;  \langle p, \zeta(p) \rangle  \le 0,  
	\end{equation} 
	then there exists $\olsi {p} \in \olsi B \cap P$ such that  $\zeta(\olsi {p}) \in P^\circ$.
\end{theoremp} 

\subsection{Retract Mapping and Supporting Lemma}
\label{rtrct-mppng}
%If the cone is not a linear vector space, the concept of retract mapping allows us to archive our purpose. 
To prepare the proofs in the following sections, we recall the concept of a retract mapping. The existence of such a mapping is shown by constructing an explicit one in \autoref{retract lemma}. This concept is used in the proof of  \autoref{Florenzano1982-thm} in cases where the cone $P $ is not a linear subspace of $\mathds R^N$. Furthermore, \autoref{spprtng-lm}, a supporting lemma, is stated and proven. 

\begin{definition}[retract mapping]
	A subspace $A$ of a topological space $X$ is a \textbf{retract of $X$} \index{retraction!retract of a set} if there is a continuous mapping $f: X \mapsto A$ such that $f(y)=y$ for all $y \in A$. The mapping $f$ is called a\textbf{ retraction of $X$} \index{retraction!map} onto $A$.  
\end{definition}

\begin{lemma} \label{retract lemma} Let $P$ be a closed convex cone in $\mathds{R}^N$. If $P \varsubsetneq span(P)$, then there exists a continuous mapping $r$ from $\olsi B \cap P $ into $S \cap P$ such that $r(x)=x, \;\;\forall x \in S\cap P$. It means that $r$ is a retract of $\olsi B \cap P $ onto $S \cap P$. 
\end{lemma} 
\noindent See \nameref{prove-retract-lm} in \autoref{proofs-sec}.% on \autopageref{prove-retract-lm}. 
%	\begin{remark} \autoref{retract lemma}  \sout{ means that if the cone $P$ is not a vector space, then the set $S\cap P$ is a retract of $\olsi  B \cap P$.}
	%	\end{remark}
\begin{remark} \autoref{retract lemma}  means that if the cone $P$ is not a vector space, then the set $S\cap P$ is a retract of $\olsi  B \cap P$.
\end{remark}

\begin{lemma} \label{spprtng-lm} Let $P$ be a closed convex cone with vertex $0_N$ in $\mathds R^N$. Let $\zeta$ be a correspondence from $\olsi  B \cap P$ into $\mathds R^N$ satisfying condition: 
	\begin{equation} \label{eq-in-sphere}
		\forall p \in S \cap P, \;\; \forall z \in \zeta(p) \;\; \langle p, z \rangle \le 0. 
	\end{equation}
	If there are some $x \in \olsi  B \cap P$ and $z \in \zeta(x)$ such that
	\begin{equation} \label{HS-ineq}
		\langle  p,z \rangle  \leq \langle x,z \rangle \;\; \forall p \in \olsi  B \cap P, 
	\end{equation}
	then $z \in P^\circ \cap \zeta(x)$.  
\end{lemma}  
\noindent See the \nameref{prove-supporting-lm} in 
\autoref{proofs-sec}.
%\hyperref[proofs-sec]{Proofs Section} 
%\ref{prove-supporting-lm} 
%on \autopageref{prove-supporting-lm}. 	  

When the correspondence $\zeta$ is single-valued, \autoref{spprtng-lm} can be  restated as follows: 
\begin{lemmap}{\ref{spprtng-lm}$^\prime$} \label{spprtng-lm2}
	Let $P$ be a closed convex cone with vertex $0_N$ in $\mathds R^N$. Let $\zeta$ be a mapping from $\olsi  B \cap P$ into $\mathds R^N$ satisfying condition: 
	\begin{equation} \label{eq-in-sphere2}
		\forall p \in S \cap P, \;\; \langle p, \zeta(p) \rangle \le 0. 
	\end{equation}
	If there is some $x \in \olsi  B \cap P$ such that
	\begin{equation} \label{HS-ineq2}
		\langle  p,\zeta(x) \rangle  \leq \langle x,\zeta(x) \rangle \;\; \forall p \in \olsi  B \cap P, 
	\end{equation}
	then $\zeta(x) \in P^\circ$.  
\end{lemmap}

\begin{remark} It is worth noting that the conclusions of \autoref{spprtng-lm} and the generalized GND lemma (\autoref{Florenzano1982-thm}) are the same. Nevertheless, the assumption in \autoref{spprtng-lm} is weaker than the one in \autoref{Florenzano1982-thm}. Indeed, first, by \autoref{HS-thm-generalized}, Condition \eqref{HS-ineq2} is a direct consequence of the upper semi-continuity of the correspondence $\zeta$.   	Second, the correspondence $\zeta$ in \autoref{spprtng-lm} is not required to be upper semi-continuous. We demonstrate this second point with the following example of a single-valued correspondence.

	% Conditions \eqref{eq-in-sphere} and \eqref{HS-ineq} hold with $x = 0$ and $z = 0$, and $0 \in P^\circ \cap \zeta(0)$. 
	
	%Normally, the class of correspondence for which conditions \eqref{eq-in-sphere} and \eqref{HS-ineq} hold is upper semi-continuous (TO PROVIDE REFERENCE). However, it is not the case in \autoref{spprtng-lm} (TO REWRITE THIS SENTENCE).
	%\red{We should show an example in which the correspondence $\zeta$ satisfying the \autoref{spprtng-lm} is not upper semi-continuous.}
\end{remark} 

\begin{example}[Example in $2-$ dimensional space]
	Consider the $2-$ dimensional cone  $P= \mathds{R}^2_+$, the closed-unit ball $\olsi B_2 = \{p=(p_1,p_2) \in \mathds R^2\mid p^2_1 + p^2_2 \le 1\}$, the unit sphere $S_2 = \{p=(p_1,p_2) \in \mathds R^2\mid p^2_1 + p^2_2 = 1\}$ and the single-valued correspondence $\zeta$ defined below 
	\begin{equation}
		\zeta(p) = 
		\begin{cases}
			-p & \text{ if } p \in P \cap \olsi B_2 \setminus (0,0),\\ 
			(-1,-1) & \text{ if } p = (0,0).  
		\end{cases} 
	\end{equation} 
	It is easy to verify that the correspondence $\zeta$ is not upper semi-continuous at $p=(0,0)$. Furthermore, Condition \eqref{eq-in-sphere2} holds, and Condition \eqref{HS-ineq2} is satisfied with $p=(0,0)$. 
	
\end{example}

%	\begin{conjecture}
	%		\label{GND-not-contintuity}
	%		Let $P$ be a closed convex cone with vertex $0_N$ in $\mathds{R}^N$. Let $\zeta$ be a non-empty, convex and compact valued correspondence from $\olsi  B \cap P$ into $\mathds{R}^N$ satisfying the condition 
	%		\begin{equation*}
		%		\forall p \in S \cap P, \;\; \forall z \in \zeta(p) \text{ such that } \langle p, z\rangle \le 0. 
		%		\end{equation*}
	%		If there exists some $x \in \olsi  B \cap P$ and $z \in \olsi  B \cap P$ such that
	%		\begin{equation*}
		%		\langle p, z \rangle \le \langle x,z \rangle \;\; \forall p \in \olsi  B \cap P,
		%		\end{equation*}
	%		then $z \in P^\circ \cap \zeta(x)$. 
	%	\end{conjecture}

%\begin{conjecture}
%	Let $P$ be a closed convex cone with vertex $0_N$ in 
%   $\mathds{R}^N$. Let $\zeta$ be a non-empty, convex and     compact valued correspondence from $\olsi  B \cap P$      into $\mathds{R}^N$ such that 
%	\begin{equation}
	%		\forall p \in S \cap P, \;\; \forall z \in \zeta(p)        \text{ such that } \langle p, z\rangle \le 0. 
	%	\end{equation} 
%  If there exists some $x \in \olsi  B \cap P$ such that    $\zeta(x) \cap P^\circ \ne \emptyset$, then there exist    $\olsi  x \in \olsi  B \cap P$ and $\olsi  z \in  \zeta(\olsi  x)$ such that 
%\begin{equation} \label{HS-ineq2}
%	\langle p, \olsi  z \rangle \le \langle \olsi  x, \olsi      z \rangle \;\; \forall p \in \olsi  B \cap P.   
%\end{equation}
%\end{conjecture}	 
%\begin{remark} Condition \eqref{HS-ineq2} holds for any 
%    compact-valued correspondence $\zeta$. Hence the upper      semi-continuity could be got rid of in the case. 
%\end{remark}

\subsection{Proof of \autoref{GNDstrong} (Gale--Nikaid{\fontencoding{T5}\selectfont \ocircumflex}--Debreu Lemma)} \label{prv-GND}
%In Section \ref{prv-GND}, we present the proof of \autoref{GNDstrong}. 
%	we split the proof into $2$ cases associating with the correspondence $\zeta$ being either single- or multi-valued. We put the proof of the case of the correspondence in the Appendix.  Below is the proof for the 
%	Section \ref{single-value-Thm2} carries out single-valued correspondence. 
%The proof for multi-valued correspondence  is shown below. 
We apply the generalized Hartman--Stampacchia theorem (\autoref{HS-thm-generalized}).\footnote{To provide intuition, we present the proof for the single-valued correspondence and include it in Appendix (Section \ref{single-value-Thm2}).} 
%To make the proof easy to read, each section below deals with  separate cases of the correspondence: it is either a single-valued or multi-valued mapping. %The proof of \autoref{GNDstrong} therefore consists of two independent parts. 
\iffalse 
\subsubsection{The Correspondence \texorpdfstring{$\zeta$}{\textzeta} is Single-valued }%\label{single-value-Thm2}
\begin{proof}\footnote{For the sake of providing intuition, we provide the proof for this case. 
		Actually, the proof for this case could be viewed as to be included in the case of correspondence.}
	In this case, we need to seek $\olsi  p \in \Delta $ such that $\zeta(\olsi  p) \in \mathds R^N_-$. Indeed, applying  \nameref{HS-thm} to the mapping $\zeta$ on $\Delta$, we obtain some $\olsi  p \in \Delta$ such that
	\begin{equation*}
		\langle p, \zeta(\olsi  p)\rangle \leq  
		\langle \olsi  p,\zeta(\olsi  p) \rangle \;\; \forall p \in \Delta.
	\end{equation*} Since the hypothesis on $\zeta$ of  \autoref{GNDstrong} (or condition  \eqref{gnd-condition}) implies $ \langle \olsi  p, \zeta(\olsi  p) \rangle \leq 0$, we see that 
	\begin{equation*}
		\langle  p, \zeta (\olsi  p) \rangle \leq 0 \;\; \forall p \in \Delta.
	\end{equation*} 
	It is obvious that this implies $\zeta(\olsi  p) \in  \mathds R^N_-. $ 
\end{proof}   
\fi 
%\subsubsection{The Correspondence \texorpdfstring{$\zeta$}{\textzeta} is Multi-valued
	%Upper Semi-continuous with Non-empty, Compact, Convex Values}\label{mlt-vl-thm2}
%{\it Proof}
%\begin{proof}  	 
Since $\zeta$ is upper semi-continuous, by applying  \autoref{HS-thm-generalized} with $C$ replaced by $\Delta$, there exist  $x\in \Delta$ and $z \in \zeta(x)$ such that
	\begin{equation*}
		\langle p, z \rangle \le \langle x,z \rangle \;\; \forall p \in \Delta.
	\end{equation*}
	Since $z \in \zeta (x)$, Condition \eqref{gnd-condition} implies that $\langle x,z \rangle \leq 0.$ Therefore,
	\begin{equation*}
		\langle p,z \rangle \leq 0 \text{ } \forall p \in \Delta. 
	\end{equation*}
	Equivalently, $z \in \mathds{R}^N_{-}$. We have proved that $z \in \zeta (x) \cap\mathds R^N _- .$
%\end{proof} \qed 

\subsection{Proof of \autoref{Florenzano1982-thm} (Generalized GND Lemma)} \label{Flo82-prf} 
Section \ref{Flo82-prf} is dedicated to proving \autoref{Florenzano1982-thm} with Condition \eqref{Florenzano1982-condition} replaced by Condition \eqref{Florenzano1982-condition2}. One of the main aims of this paper is to provide alternative and direct proofs not only for the original 
\hyperref[GNDstrong]{GND lemma} (\autoref{GNDstrong})
%\nameref{GNDstrong} 
but also for its generalized version (\autoref{Florenzano1982-thm}). As a result of this extension, the proof of \autoref{Florenzano1982-thm} comes at a cost, i.e., requiring more mathematical tools, compared with that of \autoref{GNDstrong}. Nevertheless, a concise proof of the theorem is obtained.

As a single-valued correspondence is a special case of multi-valued one, we only present the proof for the latter. For the sake of completeness, we include the proof for the former case in Appendix (Section \ref{Florenzano(82)-mapping-proof1}). Additionally, Section \ref{Florenzano(82)-mapping-proof} provides the alternative proof based solely on the geometric property of the normal cone. 	
When the cone $P$ is not a linear subspace of $\mathds R^N$, in cases where the correspondence is either multi-valued or single-valued, the proofs are based on the concept of retract mapping together with \autoref{HS-thm} and \autoref{HS-thm-generalized}, respectively. The existence results follow directly from \autoref{spprtng-lm} and \hyperref[spprtng-lm2]{Lemma} \autoref{spprtng-lm2}, respectively. When the cone $P$ is a subspace of $\mathds R^N$, in both cases, the existence results are derived from \autoref{HS-thm} and \autoref{HS-thm-generalized},  respectively.  

{\it{Proof of \autoref{Florenzano1982-thm}}} We consider $2$ cases. 
\begin{case}{\boldmath \bf{$ P \varsubsetneq span(P)$}} \label{Florenzano(1982)-correspondence-case1} 	
%	\newline 
	%\noindent The idea of the proof is .....\\
	By \autoref{retract lemma}%on page \pageref{retract lemma}
    , there is some retract mapping $r$ from $\olsi  B \cap P$ into $S\cap P$. Since  $r$ is continuous and $\zeta$ is upper semi-continuous according to the assumptions in \autoref{Florenzano1982-thm}, it follows that $\zeta \circ r$ is upper semi-continuous.  Obviously, $\zeta \circ r$ is also non-empty, convex, compact  valued.
	
	\noindent We are now applying  \autoref{HS-thm-generalized} with $C$ replaced by $\olsi  B \cap P$, $\zeta$ by $\zeta \circ r$, and thus obtain $x\in \olsi  B \cap P$ and $z \in \zeta \circ r(x)$ such that 
	\begin{equation}  \label{HS-ineq1}
		\langle  p,z \rangle \leq \langle x,z \rangle \;\; \forall p \in \olsi  B \cap P.  
	\end{equation}
	We verify that Conditions \eqref{eq-in-sphere} and \eqref{HS-ineq} in \autoref{spprtng-lm} hold with $\zeta$ replaced by $\zeta \circ r$. On one hand, it is easy to see that Condition \eqref{HS-ineq1} leads to Condition \eqref{HS-ineq} (holding  with $\zeta$ replaced by $\zeta \circ r$). On the other hand, noting that $\zeta(x^1) = \zeta \circ r (x^1)$ for all $x^1 \in S\cap P$, from the assumption on $\zeta$ of \autoref{Florenzano1982-thm} (or Condition \eqref{Florenzano1982-condition2}), we see that Condition \eqref{eq-in-sphere} of \autoref{spprtng-lm} holds. As a result of \autoref{spprtng-lm}, we obtain $z \in P^\circ \cap \zeta \circ r(x)$.  
\end{case}

\begin{case}{\boldmath \bf{ $P  = span(P)$}} \label{Florenzano(1982)-correspondence-case2}
%	\newline 
	On one hand, according to  \autoref{HS-thm-generalized} with $C$ replaced by $\olsi  B \cap P$, we obtain that there are some $x \in \olsi  B \cap P$ and $z \in \zeta(x)$ such that 
	\begin{equation*}
		\langle p, z \rangle  \le \langle x, z \rangle \;\; \forall p \in \olsi  B \cap P. 
	\end{equation*} 
	On the other hand, by the assumption on $\zeta$ of \autoref{Florenzano1982-thm}, it follows that 
	\begin{equation*}
		\forall p \in S \cap P, \;\; \forall z \in \zeta(p), \;\; \langle p, z \rangle \leq 0. 
	\end{equation*}
	\autoref{spprtng-lm} implies that $z \in P^\circ \cap \zeta(x)$. This concludes \autoref{Florenzano1982-thm}.  %\qed 
	%We repeat the procedure of Case \ref{Florenzano(1982)-correspondence-case1} on page \pageref{Florenzano(1982)-correspondence-case1} in Section \ref{Flo82-prf} with $\zeta \circ r$ replaced by $\zeta$. Thus, we get 
	%\begin{equation*}
	%z = \sum_{i=1}^{N+1}\beta_i z^i,
	%\end{equation*}
	%\begin{equation*}
	%\sum_{i=1}^{N+1}\beta_i =1,
	%\end{equation*} 
	%\begin{equation*}
	%z^i \in\zeta(x) \;\; \forall i =1, \dots, N+1. 
	%\end{equation*}
	%and 
	%\begin{equation*}
	%z \cdot p \le z \cdot x \;\; \forall p \in \olsi  B \cap P. 
	%\end{equation*}
	%We consider the following $2$ subcases: $||x||=0$ or $||x||=1$, and $0 < ||x||<1$. We now apply the similar argument as in Case \ref{Florenzano(1982)-mapping-case2} on page \pageref{Florenzano(1982)-mapping-case2} of Section \ref{Florenzano(82)-mapping-proof} with $\zeta(\olsi  p)$ replaced by $z$, $\olsi  p$ by $x$. We thus get $z\in \zeta(x)\cap P^0$. This concludes \autoref{Florenzano1982-thm}
\end{case}
\qed

\iffalse 
\begin{remark} When the correspondence $\zeta$ is either single-valued or multi-valued, there is another approach to \autoref{Florenzano1982-thm}. To demonstrate the method, we focus on the multi-valued correspondence. Instead of using only one retract mapping as in \autoref{Florenzano(1982)-correspondence-case1} 
%Case  \ref{Florenzano(1982)-correspondence-case1} 
on page \pageref{Florenzano(1982)-correspondence-case1}, a family of retract mapping $(r_a)_a$ is used. See details of proof in Section  \ref{Apendix-proof} on page \pageref{Apendix-proof}.  
\end{remark}
\fi 

\iffalse 
\begin{remark}\label{remark0} 
As the result of the direct proof  of \autoref{Florenzano1982-thm} given , we have the result that, when $P$ is not a linear subspace, there exists $\olsi  p \in S \cap P$ such that $\zeta (\olsi  p) \cap P^0 \ne \emptyset$, while in the statement of \autoref{Florenzano1982-thm}, we have only $\{\exists \olsi  p \in B\cap P, \zeta (\olsi  p)\cap P^0 \neq \emptyset\}$ (a priori $\olsi  p$ may be zero). However, %Florenzano and Le Van 
\cite{FlorenzanoLeVan1986} exhibit a corollary of \autoref{Florenzano1982-thm} which ensures $\olsi  p \neq 0$.
\end{remark}
\fi

\begin{remark}\label{remark0} 
In the statement of \autoref{Florenzano1982-thm}, it is possible that $\olsi  p$ equals $0_N$. It is perhaps worth mentioning that in our direct proof of \autoref{Florenzano1982-thm}  when the cone $P$ is not a linear subspace of $\mathds R^N$, we conclude that $\olsi  p$ is different from $0_N$. This result is not demonstrated in \cite{Florenzano1982}; it is also obtained in Corollary $3$ of \cite{FlorenzanoLeVan1986} using a different approach.   
\end{remark}

%\begin{remark} We are able to imply the Kakutani fixed-point theorem from GND lemma. \end{remark}
\section{Equivalence Relations} \label{Equivalnce-relation}
\subsection{The Hartman--Stampacchia and Brouwer Fixed-Point Theorems} \label{HS-Brouwer}
In \hyperref[HS-Brouwer]{Section \ref{HS-Brouwer}}, for completeness, we recall the implication demonstrating that the Brouwer fixed-point theorem follows directly from the Hartman--Stampacchia theorem. 
%Although the proof for  \autoref{Brouwer-thm}  could be found in ....., we include it for the sake of completeness. 
\begin{proposition}[Brouwer fixed-point theorem]\label{Brouwer-thm} Let $C$ be a non-empty, convex, compact set in $\mathds R^N$. Let $f$ be a continuous mapping from $C$ into itself. Then, there exists a fixed point of $f$.
\end{proposition} 
{\it Proof}
%\begin{proof}
%Define the mapping $H$ by $H(x)= \{y\in\mathds R ^N: y_i = \max (x_i, 0)\}$. We have $H(x) \geq 0, \forall x$ and $H(x)= 0 \Leftrightarrow x_i \leq 0, \forall i$. Obviously $H$ is continuous from $\Delta$ into $\mathds R^N$.\\
Define $g(x)= f(x)-x$. Applying 
\autoref{HS-thm} 
to the mapping $g$, we obtain some $\olsi  x \in C$ such that
$$\langle p, g(\olsi  x) \rangle \le \langle \olsi  x, g(\olsi  x) \rangle \;\; \forall p \in C.$$
%	$$\langle g(\olsi  x), p-\olsi  x \rangle \leq 0,\;\; \forall p \in C.$$
We claim that $\olsi  x$ is a fixed point of $f$. Indeed, take $p= f(\olsi  x) \in C$. Then
$$\langle f(\olsi  x)-\olsi  x,  f(\olsi  x)-\olsi  x \rangle  \leq 0.$$
In other words, $f(\olsi  x)= \olsi  x.$
%\end{proof}
\qed 
\begin{remark} We can use the Brouwer fixed-point theorem to prove the Hartman--Stampacchia theorem \cite{KinderlehrerStampacchia2000}. Indeed, let $g(x)= \pi _C (x+\zeta(x))$ for any $x \in C$, where $\pi_C$ denotes the convex projection of $\mathds R^N$ onto $C$. The mapping $g$ is continuous from $C$ into $C$. From \autoref{Brouwer-thm}, there is a fixed-point of $g$, i.e.,  $\olsi  x = g(\olsi  x)$ or equivalently $\olsi  x = \pi _C (\olsi  x + \zeta(\olsi  x))$. In this case, $\zeta(\olsi  x)= \olsi  x + \zeta(\olsi  x)-\olsi  x$ and this expression belongs the normal cone of $C$ at $\olsi  x$. We get Inequality (\ref{HS inequality}) of the Hartman--Stampacchia theorem.
\end{remark}

\subsection%[Fixed Point Theorems]
{The Hartman--Stampacchia and Kakutani Fixed-Point  Theorems}\label{prove-Brouwer-Kakutani} 
In \hyperref[prove-Brouwer-Kakutani]{Section \ref{prove-Brouwer-Kakutani}}, we present a proof of the Kakutani fixed-point theorem using the  \hyperref[HS-thm]{Hartman--Stampacchia theorem},  
as stated in \autoref{Kakutani-thm}.

\begin{proposition}[Kakutani fixed-point theorem]\label{Kakutani-thm}
Let $C$ be a non-empty, convex, compact, subset of $\mathds R^N$. Let $\zeta$ be a non-empty, convex, compact,  valued correspondence from $C$ into itself. If $\zeta$ is an upper semi-continuous, then there exists a fixed point of the correspondence $\zeta$. That is, there exists some $x \in C$ such that $x \in \zeta(x)$. 
\end{proposition}
{\it Proof of \autoref{Kakutani-thm} using the Hartman-Stampacchia theorem}. 
%\begin{proof}[Proof of \autoref{Kakutani-thm} using Hartman-Stampacchia theorem] 
%$\quad \quad \quad \quad \quad \quad \quad \quad \quad \quad \quad \quad \quad \quad \quad \quad \quad \quad$ 
Let $(\varepsilon_k)$ be a decreasing non-negative sequence converging to $0$. By \autoref{representaion-extra}, for any $k \in \mathds N^*$, there is some continuous mapping $f^k$ satisfying \hyperref[representation00]{Condition \ref{representation00}}. Applying the Hartman--Stampacchia gives $x^k$ such that 
\begin{equation*}
	\langle p - x^k, f^k(x^k)-x^k \rangle  \le 0 \;\; \forall p \in C. 
\end{equation*}
Substituting $p$ for $f^k(x^k)$ into the above inequality implies $f^k(x^k) = x^k$. By  \autoref{representaion-extra} again, there exist at most $N+1$ vectors $z^{1,k}\dots, z^{N+1,k}$ in $\zeta \big( B(x^k,\varepsilon_k)\big)$ and strictly positive numbers $\beta_{1,k}, \dots, \beta_{N+1,k}$ such that 
\begin{equation}   \label{representaion f^k 3}
	f^k(x^k) = \sum_{i=1}^{N+1} \beta_{i,k}z^{i,k}, 
\end{equation}
with $\sum_{i=1}^{N+1}\beta_{i,k}=1$. 
We now use the compactness argument to prove the existence of a fixed point of $\zeta$. Indeed, note that there exists $u^{i,k}\in \olsi  B$ such that $z^{i,k} \in \zeta(x^k+ \varepsilon_k u^{i,k})$ for any $i=1, \dots, N+1$ and $k \in \mathds N^*$. Since  $\Big(\big(x^k, (\beta_{i,k})_{i=1}^{N+1}, (u^{i,k})_{i=1}^{N+1}\big)\Big)_{k}$ is a sequence in a compact set $C \times [0,1]^{N+1}\times \olsi  B^{N+1}$, without loss of generality, we might assume that the sequence converges to $\big(x, (\beta_i)_{i=1}^{N+1}, (u^i)_{i=1}^{N+1}\big)$. For any $i=1,\dots, N+1$, since $\lim\limits_{k \to \infty} \big(x^k + \varepsilon_k u^{i,k}\big)= x$, by the compactness of $\zeta(x)$ and the upper semi-continuity of $\zeta$, we conclude that there are some $z^i\in C$ and there exists a strictly increasing subsequence $(k_n)_n\subset (k)$ such that $\lim\limits_{n \to \infty} z^i_{k_n} = z^i$ and $z^i \in \zeta(x)$.  It is obvious that $\sum_{i=1}^{N+1}\beta_i=1$. The convexity of  $\zeta(x)$ implies $\sum_{i=1}^{N+1}\beta_i z^i \in \zeta(x)$. As proved above, $x^k$ is the fixed point of $f^k$ and $\lim\limits_{n \to \infty } x^{k_n} = x$, implying that 
%. It follows from the LHS of identity \eqref{representaion f^k 3} that 
\begin{equation*}
	\lim\limits_{n \to \infty}f^{k_n}(x^{k_n}) = x.
\end{equation*}
On the other hand, the convergences of $\{\beta_{i,k_n}\}$ and $\{z^i_{k_n}\}$ imply that
\begin{equation*} 
	\lim\limits_{ n \to \infty} \sum_{i=1}^{N+1} \beta_{i,k_n} z^i_{k_n} 
	= \sum_{i=1}^{N+1} \beta_i z^i. 
	%\in \zeta(x)
\end{equation*}
Combining the above convergences with Identity \eqref{representaion f^k 3} proves $x= \sum\limits_{i=1}^{N+1} \beta_i z^i$. 
Since the set $\zeta(x)$  is convex and $z^i \in \zeta(x)$ for $i=1, \dots, N+1$, it follows that $\sum\limits_{i=1}^{N+1} \beta_i z^i \in \zeta(x)$. As a result, $x \in \zeta(x)$. This concludes the existence of a fixed point of $\zeta$.
%	Then we repeat the procedure of the proof of %\nameref{proof using Brouwer} 	\autoref{Kakutani-thm} 	on page \pageref{proof using Brouwer} and conclude that  there exists a fixed point of $\zeta$. 
%\end{proof} 
\qed 
\begin{remark} As in \cite{Cellina1969a}, the so-called fixed point of $\zeta$ is constructed by a cluster point of a set of fixed points of \q{approximate mappings} to $\zeta$. The cluster point is shown as a fixed point by distance between it and the graph of the correspondence $\zeta$  being as small as one wants. As in our approach, we exploit the convexity of the correspondence.
\end{remark}
\begin{remark} We could obtain the proof for the Kakutani fixed-point theorem by using the Brouwer's fixed-point argument and the same procedure as above. 
\end{remark} 

\section{Conclusion}

In this paper, we extend the Hartman-Stampacchia theorem to convex correspondences; we use both the original theorem and its generalized version as primary tools to prove not only the GND lemma but also its generalized version, without relying on a fixed-point argument. Furthermore, we establish the cycle of equivalences among the GND lemma, the Hartman--Stampacchia theorem, its generalized form, and the Kakutani and  Brouwer fixed-point theorems. As presented in the proof of \autoref{HS-thm-generalized} and later applied to the generalized GND (\autoref{Florenzano1982-thm}), we hope the proof 
potentially provides a novel approach to studying the existence of equilibrium, particularly 
an alternative numerical approach to computing equilibrium prices, compared to the approaches by Scarf \cite{Scarf82} and  Scarf and Hansen \cite{ScartHansen73},  and the approach proposed by 
Le et al. \cite{Le2022Sperner}. 

This paper has focused on finite-dimensional spaces. An interesting direction for future research would be to extend our  analysis to infinite-dimensional settings \cite{CornetETB2020,CornetGouYannelis2023,KhanMcLeanUyanik_ET2025,Otsuka_ETB2024,YannelisJMAA1985}.  Another promising avenue is to investigate whether the continuity assumption in the Hartman--Stampacchia theorem can be relaxed, as well as to explore the potential applications of such a relaxation in establishing the existence of equilibria under discontinuities \citep{CornetETB2020, CornetGouYannelis2023, HeYannelis_ET2016, HeYannelis_JMAA2017, KhanMcLeanUyanik_ET2025, PodczeckYannelis_ET2022, PodczeckYannelis_ET2024}.
%Reviewer #2: The paper seems like a solid piece of work. However, no credit is given to earlier work on the GND Lemma. For example the weak Walras law was used in Yannelis JMAA 85, and Krasa-Yannelis ET 1994, together with the upped demi continuity of the excess demand correspondence ( weaker than use). Also there are related papers that relax the continuity of the excess demand, e.g., He-Yannelis JMAA and ET 2017/2018 , Cornet ETB 2020 and others...The recent paper of Khan et all in ET 2025  is not mentioned. I am glad to accept a revised version of this paper but credit should be given to earlier contributions on the subject. The references should be updated and perhaps add open questions. It is not clear how this work could be extended to infinite dimensional spaces and also allow for discontinuous exceed demands.

\begin{acknowledgements} We thank Jean-Marc Bonnisseau, Antonio D’Agata and anonymous referees for their helpful comments and discussions.
%Acknowledgements to sponsoring agencies and individuals should be placed here.
\end{acknowledgements}
 
\appendix  %This command ends the counting of sections.

\renewcommand{\theequation}{\arabic{equation}}  % reset to default numbering style
\setcounter{equation}{0}  % Reset equation numbering to 1

\section{Appendix} \label{proofs-sec}
\subsection{Proof of \autoref{representation-lemma}}
\label{prove-representation-lemma}
%	\begin{proof}[Proof of \autoref{representation-lemma}]
	We construct the mapping $f$ in $2$ steps. First, we use a partition of unity subordinated to a covering of $C$. Second, we verify that the resulting mapping satisfies  \hyperref[representation00]{Condition \ref{representation00}}. 
	
	\noindent 
	By the compactness of $C$, there exists a finite covering of $C$, say $\big(B(x^i,r)\big)_{i=1}^M$ where $B(x^i,r)$ are open balls of radius $r$ centered at point $x^i\in C$, for $i=1,\dots,M$. Let $\big(\alpha_i\big)_{i=1}^M$ be a partition of unity subordinated to the covering $\big(B(x^i,r)\big)_{i=1}^M$  (See Section $2.19$ in Aliprantis and Border \cite{AliprantisBorder2006} on page $66$ for a detailed explanation partition of utility). Take $y^i \in \zeta(x^i)$ for all $i=1, \dots ,M$. We define \footnote{
		A very similar idea, using the convex combination with coefficients being a partition to a covering in \autoref{representation-lemma}, can be found in \cite{Cellina1969a}. However, it is used for different purposes. While  Theorem $1$ in \cite{Cellina1969a} is used to build a continuous mapping whose graph is separated with the correspondence with a small distance, in our paper it is used to extend the Hartman--Stamppachia's theorem to correspondence.
	} the mapping $f: C \to \mathds{R}^N$ by
	\begin{equation*}
		f(x) = \sum_{i=1}^M \alpha_i(x)y^i, \quad 
		\forall x\in C. 
	\end{equation*} 
	It is clear that $f$ is continuous on $C$ since it is a finite sum of continuous mappings. 
	
	\noindent Next, we prove that the mapping $f$ satisfies  \hyperref[representation00]{Condition \ref{representation00}}. Indeed, fix some $x \in C$. Let $J = \{i \in\mathds{N}: 1 \le i \le M \text{ and } x \in B(x^i,r)  \}$. Observe that $ x \in \cap_{i\in J} B(x^i, r)$, so $ x^i \in B( x, r)$ and $y^i \in \zeta (B( x, r))$ for all $i \in J$.  If $i \notin J$, then $x \notin B(x^i,r)$, and hence the partition of unity ensures that $\alpha_i(x) = 0$. Thus, we have
%	Note that if $i \notin J$, it follows that 	$x \notin B(x^i,r)$ and thus that the partition of unity over C subordinated to the covering $\big(B(x^j,r)\big)_{j=1}^M$ implies $\alpha_i(x)=0$. Consequently, 
	\begin{equation*}
		f(x) =\sum _{i=1} ^{M} \alpha_i (x) y ^i = \sum _{i \in J} \alpha_i (x) y ^i + \sum _{i \notin J} \alpha_i (x) y ^i =\sum _{i \in J} \alpha_i (x) y ^i. 
	\end{equation*}
	Since $\sum_{i \in J} \alpha_i(x) =1$, we conclude that  $f(x) \in co\Big(\zeta  \big(B( x, r)\big)\Big)$, where $co(S)$ denotes the convex hull of the set $S$. By the Carath\'eodory convexity theorem  \cite{Caratheodory1907}  which states that in an $n-$dimensional vector space, every vector in the convex hull of a nonempty set can be written as a convex combination of at most $(n+1)$ vectors from the set \footnote{For a simple proof, see Proposition $1.1.2$ in  \cite{FlorenzanoLeVan2001}  or Theorem $5.32$ in \cite{AliprantisBorder2006}.}, there exist at most $N+1$ vectors $z^1,\dots,z^{N+1} \in \zeta\big(B( x, r) \big)$ and strictly positive numbers $\beta_1,\dots,\beta_{N+1}$ such that 
	$$ f (x)  =\sum _{i=1}^{N+1} \beta _i z^i \quad \text{ with } \quad \sum _{i=1}^{N+1} \beta_i=1.$$ 
	This completes the proof. 
	%	\end{proof} 
%\qedsymbol 	
\subsection{Proof of \autoref{HS-thm-generalized}} \label{proof-HS-thm-generalized} 
Let $(\varepsilon_k)_k$ be a non-negative sequence converging to $0$. For any $k \ge 1$,  apply \autoref{representation-lemma} 
with $r=\varepsilon_k$ to obtain a continuous mapping $f^k: C \to \mathds{R}^N$ that satisfies  \hyperref[representation00]{Condition \ref{representation00}}. Then, by applying  
\hyperref[HS-thm]{Hartman--Stamppachia} theorem to the mapping $f^k$ on $C$, there exists  some $x^k\in C$ such that the following inequality holds: 
\begin{equation} \label{HS ineq 2} 
	\langle p, f^k(x^k)\rangle \le \langle x^k,f^k(x^k)\rangle  \quad \forall p \in C.
\end{equation}
From \autoref{representation-lemma}, we know that there exist\footnote{Upper indices mark vectors and lower indices real numbers.} at most $(N+1)$ vectors $z^{1,k},\dots,z^{N+1,k} \in \zeta\big(B(x^k, \varepsilon_k)\big)$ and positive numbers $\beta_{1,k}\dots, \beta_{N+1,k}$ such that
\begin{equation} \label{representation f^k 1} 
	f^k (x^k)  =\sum _{i=1}^{N+1} \beta_{i,k} z ^{i,k}  
\end{equation}
with $\sum _{i=1}^{N+1} \beta_{i,k}=1$. For each $i=1, \dots, N+1$ and $k \in \mathds N^*$, since $z^{i,k} \in \zeta\big(B(x^k, \varepsilon_k)\big)$, there exists some $u^{i,k}$ in closed unit ball $\olsi  B$ such that $z^{i,k} \in \zeta\big(x^k+\varepsilon_k u^{i,k}\big)$. Observe that $\Big( \big(x^k, (\beta_{i,k})_{i=1}^{N+1}, (u^{i,k})_{i=1}^{N+1} \big)\Big)_{k}$ is a sequence in the compact set $C \times [0,1]^{N+1} \times \olsi  B^{N+1}$. So,  without loss of  generality, we can assume that this sequence converges to some limit $\big(x, (\beta_i)_{i=1}^{N+1}, (u^i)_{i=1}^{N+1} \big)$. Since $\lim\limits_{ k \to \infty}\varepsilon_k = 0$, it follows that $\lim\limits_{k \to \infty} x^k + \varepsilon_k u^{i,k} = x$. In addition, since $\zeta(x)$ is compact and $ z^{i,k}  \in \zeta(x^k+\varepsilon_k u^{i,k})$, the upper semi-continuity of $\zeta$ implies
%\footnote{See more about the semi-continuity property in \autoref{semi-continuity criterion} in Appendix on page \pageref{semi-continuity criterion}.} 
that there is a subsequence $\big(z^{i,{k_n}}\big)_{k}$ being convergent to $z^i \in \zeta(x)$ for all $i=1, \dots,N+1$. 
It is obvious from Identity \eqref{representation f^k 1} that
\begin{equation} \label{convergence}
	\lim\limits_{n \to \infty} f^{k_n} (x^{k_n}) = \sum _{i=1}^{N+1}  \beta_i z^i:=z. 
\end{equation}
Since $z^i \in \zeta(x)$ and $\sum_{i=1}^{N+1} \beta_i=1$, the convexity of $\zeta(x)$ ensures $z \in \zeta(x)$. 
%and $z \in \zeta(x)$ since $\sum_{i=1}^{N+1} \beta_i=1$ and  $\zsec:eta(x)$ is convex.   
Finally, combining Inequality \eqref{HS ineq 2}  with  Identity \eqref{representation f^k 1}  and Convergence \eqref{convergence}, we obtain 
\begin{equation*}
	\langle p, z \rangle  \leq  \langle x,z \rangle  \quad \forall p \in C. 
\end{equation*}   
This completes the proof of \autoref{HS-thm-generalized}. 
%\end{case}  
%\end{proof} 
%\qedsymbol 
\subsection{Proof of  \autoref{HS-thm-generalized-lwr}} \label{proof-HS-thm-generalized-lwr}
%	\begin{proof}
	%\begin{case}{\bf{$\boldsymbol{\zeta}$} is upper semi-continuous}
	Indeed, since $\zeta$ is lower semi-continuous, due to  \cite{Michael1956} (Theorem $3.1^{\prime\prime\prime}$), we obtain\footnote{This is a particular case of Theorem $3.1^{\prime\prime\prime}$ in \cite{Michael1956}. For detailed proof, see, e.g., Proposition 10 in \cite{Florenzano1981} or Proposition 1.5.3 in \cite{Florenzano2003} on page 31.
		%or \autoref{selection-thm} in Appendices on \autopageref{selection-thm}
	}  a continuous selection mapping $f$ of $\zeta$. 
	%there exists a continuous selection $f$ of $\zeta$. 
	Applying \autoref{HS-thm} to mapping $f$ with $K$ replaced by $C$, we obtain  
	$x \in C$ such that
	\begin{equation}
		\langle p, f(x) \rangle \leq \langle x, f(x)\rangle  \;\; \forall p \in C. 
	\end{equation}
	%$$ \langle f(\olsi  x), v-\olsi  x \rangle  \leq 0, \;\; \forall v \in C.$$
	Since $f$ is a selection mapping of $\zeta$, it follows that $f(x) \in \zeta(x)$. 
	Define $z= f(x)$ to end the proof. 
	%. Since $f$ is a continuous selection mapping of $\zeta$, it follows that $f(x) \in \zeta(x)$. 
	%$\zeta (\olsi  x)$ to end the proof.
	%\end{case} 
	%	\end{proof}  
%\qedhere \qed 
%\qedsymbol
\subsection{Proof of \autoref{retract lemma}}  
\label{prove-retract-lm}
%	\begin{proof}[Proof of \autoref{retract lemma}] 
	The proof follows the idea of \cite{FlorenzanoLeVan1986} (see Corollary $3$ in \cite{FlorenzanoLeVan1986}). 
	
	\noindent Since $P \varsubsetneq span(P) $, then $P$ is not a subspace of $\mathds{R}^N$, consequently, there exists $a \in -P\backslash P$. We show in %Claim
	\hyperref[non-empty polar cone and -cone]{Claim}
	%\autoref{non-empty polar cone and -cone}  
	below that it is possible to choose  such $a$ satisfying $a \in P^0$  (see \autoref{exist-vector-a}). 
	\begin{claim} \label{non-empty polar cone and -cone} There is some $a \in P^\circ  \cap (-P) \cap P^{\mathsf{c}}$ and $a \ne 0_N$. 
	\end{claim} 
	\begin{figure}[ht]
		\centering 
		%		\begin{center}
			\begin{tikzpicture}[domain=-100:100,scale = 0.03]
				%[scale = 0.1] 
				\newcommand{\angA}{20}
				%		\newcommand{\a}{30}
				%		\newcommand{\angA}{\tan}
				
				%axis 
				
				\coordinate (O) at (0,0); 
				\coordinate (X1) at (100,0); 
				\coordinate (X2) at (-100,0);
				\coordinate (Y1) at (0, 100);
				\coordinate (Y2) at (0, -100);
				\draw (X1)--(X2); 
				\draw (Y1)--(Y2); 
				
				\draw[black,domain=-100:100] plot(\x, {tan(\angA)*\x});   
				\draw[black, domain=-35:35] plot(\x, {(tan(90-\angA))*\x}); 
				%		\draw (O) circle (78.9491cm); 
				
				\draw[blue, latex-latex] (20:78.9491) arc(20:70:78.9491) node[midway,right, scale =.7]{$P$};   % a= 50, the radius = 78.9491 
				
				\draw[blue, latex-latex] (160:99.7390) arc(160:290:99.7390) node[pos =.4, left,scale=.7]{$P^\circ$}; % b_parameter= 60 the radius = 94.7390
				
				\draw[blue, latex-latex] (200:78.9491) arc(200:250:78.9491) node[pos =0.5,below, scale =.7]{ $-P$};    % b_parameter = 50; radius = 78.9491 
				
				\draw[->, red] (0,0) -- (-20,-25) node[red, below, scale=.7]{$a$};

				\draw[black, domain =-100:0] plot(\x, {(tan(180-\angA))*\x} ); 
				
				\draw[black, domain=0:35] plot(\x, {tan(270+\angA)*\x});
				%		\draw[black,domain=-30:40] plot(\x, {(tan(90-\angA))*\x}); 
				%		\draw[domain=0:30] plot(\x, {(tan(270+\angA))*\x}); 
				%		\draw[domain=-80:0] plot(\x, {(tan(180-\angA))*\x}); 
				
			\end{tikzpicture}
			\caption[short text]{Illustrate the vector $a$.}
			\label{exist-vector-a}
			
		\end{figure}
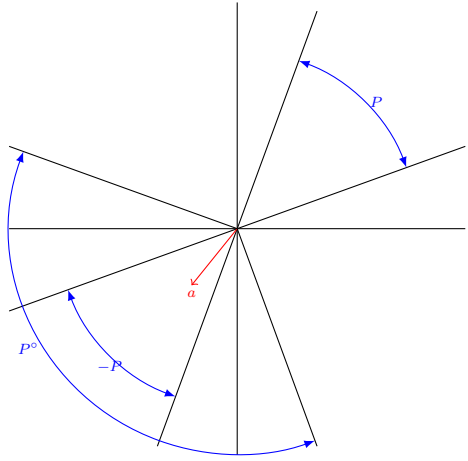
		{\it Proof of \hyperref[non-empty polar cone and -cone]{Claim}}
%		\begin{proof}[Proof of			\autoref{non-empty polar cone and -cone}] 
			Since $P$ is not a subspace, there is some $ x \in P$, but $x \notin -P$. Define $y$ to be the orthogonal projection of $x$ onto $-P$. Let $a = y + (-x)$. Since $x \notin -P$ and $-P$ is closed, it follows that $a \ne 0_N$. On one hand, because $-P$ is a convex cone, $y$ and $-x$ belong to $-P$, hence $a$ belongs $-P$. On the other hand, by the choice of $y$, 
			\[ \langle y- x, y - \olsi  y \rangle \le 0  \text{ for all } \olsi  y \in -P. \]
			Note that $\olsi  y = -nz \in -P$ with $n>0$ and $z \in P$, and that $a=y-x$. Substituting them into the above inequality, we get   
			$$\langle a, y + nz\rangle \le 0.$$
			This leads to
			\begin{equation*}
				\langle a,z \rangle \le -\frac{1}{n} \langle a, y\rangle \;\; \text{ for all } n>0\;\text{ and } z  \in P.
			\end{equation*}
			Letting $n$ go to infinity proves that$\langle a, z \rangle \le 0$ for all $z \in P$. In other words, $a \in P^\circ $. Furthermore, since $a \in P^\circ $ and $\langle a,a\rangle > 0$, we can deduce that $a \notin P$ or $a \in P^{\mathsf{c}}$. The proof of 
			\hyperref[non-empty polar cone and -cone]{Claim}
			%Claim 
%			\autoref{non-empty polar cone and -cone} is over. %\phantom\qedhere
%		\end{proof} 
\qed
		\noindent Now we construct a retract mapping $r$. 
		According to %Claim  
		\hyperref[non-empty polar cone and -cone]{Claim},  there exists  $a \in P^\circ  \cap (-P) \cap P^{\mathsf{c}}$ and $a \ne 0_N$. Fix $x\in \olsi  B \cap P$. Consider the following equation with some real variable $\lambda_a(x)$ (see \autoref{exist-lambda}):
		\begin{figure}[ht] 
			\centering 
			%		\begin{center}
				\begin{tikzpicture}[domain=-5:100,scale = 0.05]
					%[scale = 0.1] 
					\newcommand{\angA}{20}
					%		\newcommand{\a}{30}
					%		\newcommand{\angA}{\tan}
					
					%axis 
					
					\coordinate (O) at (0,0); 
					\coordinate (X1) at (100,0); 
					\coordinate (X2) at (-50,0);
					\coordinate (Y1) at (0, 100);
					\coordinate (Y2) at (0, -50);
					\draw (X1)--(X2); 
					\draw (Y1)--(Y2); 
					
					\draw[black,domain=0:100] plot(\x, {(tan(\angA))*\x});  
					
					\draw[black, domain=0:35] plot(\x, {(tan(90-\angA))*\x}); 
					%\draw (O) circle (78.9491cm); %b_parameter = 50
					
					\draw[latex-latex] (20:102.6339) arc(20:70:102.6339)  node[midway,right]{P};    
					%\draw (O) circle (102.6339cm);
					% radius = 102.6339 = a/cos(1-\angA), a_paratemer= 65 
					
					\draw (20:78.9491) arc(20:70:78.9491); 
					% b_parameter = 50 
					
					%			\draw[latex-latex] (160:94.7390) arc(160:290:94.7390) node[midway,left]{$P^\circ$}; 
					
					%			\draw[latex-latex] (200:78.9491) arc(200:250:78.9491) node[midway,left]{\footnotesize -P};    
					\newcommand{\xa}{-5}
					\newcommand{\ya}{-40}
					\newcommand{\xx}{30}
					\newcommand{\yx}{25}
					\draw[->] (0,0) -- (\xa,\ya) node[left, scale =.7, pos =0.5]{$ a$}; 
					\draw[->] (0,0) -- (\xx,\yx) node[pos=0.6,below, scale=0.7]{$x$}; 
					\draw[blue, dashed] (\xa,\ya) -- (\xx,\yx); 			
					%			\coordinate (C) at (61.4161,49.6087);
					\coordinate (C) at (49.6087,61.4161);			
					\draw[red, ->] (\xx,\yx)--(C) node[pos =0.2,right, red, scale=.7]{ $\lambda_a(x)(x-a)$};  			
					\draw[blue, ->] (O)--(C) node[pos =0.7,left, scale=0.7]{$ r(x)$}; 		
					%			\draw[black, domain =-100:0] plot(\x, {(tan(180-\angA))*\x} ); 
					
					%			\draw[black, domain=0:35] plot(\x, {tan(270+\angA)*\x});
					%		\draw[black,domain=-30:40] plot(\x, {(tan(90-\angA))*\x}); 
					%		\draw[domain=0:30] plot(\x, {(tan(270+\angA))*\x}); 
					%		\draw[domain=-80:0] plot(\x, {(tan(180-\angA))*\x}); 
					
				\end{tikzpicture}
				\caption[short text]{Illustrate the variable $\lambda_a(x)$.}
				\label{exist-lambda} 
			\end{figure}
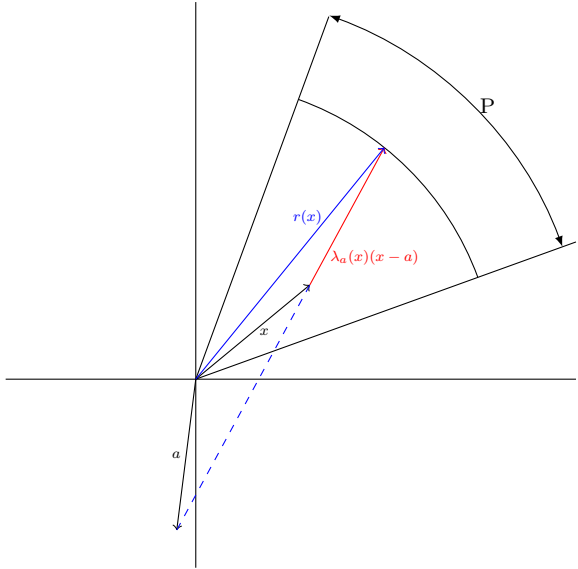  
			\begin{equation}
				||x + \lambda_a(x) (x-a) || =1.
			\end{equation}
			This leads the quadratic equation: 
			\begin{equation}\label{retract-eq-lambda}
				||x-a||^2 \lambda^2_a(x) + 2\langle x, x-a \rangle \lambda_a(x) + ||x||^2-1 = 0.
			\end{equation} 
			Since $||x-a||^2 \left( ||x||^2-1 \right) \le 0$ and $a\not=x$, this quadratic equation has at least one non-negative solution. We are able to compute an explicit formula for this solution as follows: 
			\begin{equation} \label{retract-lambda}
				\lambda_a(x)=  \frac{-\langle x, x-a \rangle + \sqrt{\langle x, x-a\rangle^2 + (1-||x||^2)||x-a||^2}} {||x-a||^2}.
			\end{equation} 
			Let us define the mapping $r$ by\footnote{By the construction, the retract $r$ is dependent on the vector $a$. 
				%This is later used for alternative proof of \autoref{Florenzano1982-thm}, see Section  \ref{Apendix-proof} in Appendix on page \pageref{Apendix-proof} for more details.
			}
			\begin{equation} \label{retract} 
				r(x) = x+ \lambda_a(x)(x-a). 
			\end{equation}
			On one hand, by the construction, $||r(x)||=1$. On the other hand, $r(x)$ can be alternately described as $r(x) = \big(1+ \lambda_a(x)\big)x + \lambda_a(x)(-a)$. Because $x$ and $-a$ are in the convex cone $P$ and $\lambda_a(x)\ge 0$, it follows that $r(x)$ belongs $P$. Therefore, we have constructed the well-defined mapping $r$ from $\olsi  B \cap P$  to $S \cap P$. Since $\lambda_a(x)$ is continuous with respect to $x$ on $\olsi  B \cap P$, then so is the mapping $r$. 
			To end the proof, it remains to show that $\displaystyle r_{\mid S\cap P} = id_{S \cap P}$. Indeed, consider $x \in S \cap P$, then $||x||=1$. Since $a \in P^\circ$ and $x \in P$ then $\langle x,a\rangle \le 0$ implying $\langle x, x-a \rangle \ge 0$. From \eqref{retract-lambda} we get  $\lambda_a(x)=0$. Consequently,  \eqref{retract} leads to $r(x)=x$. 
			%	\end{proof}  
%		\qedsymbol	
		\subsection{The Proof of \autoref{GNDstrong}: when \texorpdfstring{$\zeta$}{\textzeta} is Single-valued}\label{single-value-Thm2}
		We directly apply the
		\hyperref[HS-thm]{Hartman-Stampacchia theorem}
		%\nameref{HS-thm} 
		to this case.\footnote{For the sake of providing intuition, we provide the proof for this case. Actually, the proof for this case could be viewed as to be included in the case of correspondence.}  
		%{\it Proof}
%		\begin{proof}
			We need to seek $\olsi  p \in \Delta $ such that $\zeta(\olsi  p) \in \mathds R^N_-$. Indeed, applying  the 
		\hyperref[HS-thm]{Hartman-Stampacchia theorem}
%		\nameref{HS-thm} 
		to the mapping $\zeta$ on $\Delta$, we obtain some $\olsi  p \in \Delta$ such that
			\begin{equation*}
				\langle p, \zeta(\olsi  p)\rangle \leq  
				\langle \olsi  p,\zeta(\olsi  p) \rangle \;\; \forall p \in \Delta.
			\end{equation*} Since the assumption about $\zeta$ in \autoref{GNDstrong} (or Condition  \eqref{gnd-condition}) implies $ \langle \olsi  p, \zeta(\olsi  p) \rangle \leq 0$, we see that 
			\begin{equation*}
				\langle  p, \zeta (\olsi  p) \rangle \leq 0 \;\; \forall p \in \Delta.
			\end{equation*} 
			It is obvious that this implies $\zeta(\olsi  p) \in  \mathds R^N_-. $ 
%		\end{proof}    
\qed 
		\subsection{The Proof of \autoref{Florenzano1982-thm}: when  \texorpdfstring{$\zeta$}{\textzeta} is Single-valued}  \label{Florenzano(82)-mapping-proof1}
%		{\it Proof}
%		\begin{proof}   
			In this case, we need to seek $\olsi  p \in \olsi  B \cap P$ such that $\zeta (\olsi  p) \in P^\circ$. The proof splits into $2$ separate cases:
			\setcounter{case}{0}
			\begin{case} {\boldmath \bf{$ P \varsubsetneq span(P)$}} 
%			\newline 
				According to the assumption on $\zeta$, the mapping $\zeta$ is continuous. By \autoref{retract lemma} 
                %on page \pageref{retract lemma}
                , there is some retract $r: \olsi  B \cap P \to S \cap P $. Since $\zeta$ and $r$ are continuous on $\olsi  B \cap P$, it follows that so is the mapping $\zeta \circ r$. We apply \autoref{HS-thm} to the mapping $\zeta \circ r$ on $\olsi  B \cap P$, and thus obtain some $\olsi  x \in \olsi  B \cap P$  such that
				\begin{equation} \label{HS ineq 1}
					\langle p, \zeta \circ r (\olsi  x) \rangle  \leq \langle  \olsi  x, \zeta \circ r (\olsi  x)\rangle 
					\;\; \forall p \in \olsi  B \cap P. 	
				\end{equation}
				We now deploy \autoref{spprtng-lm} with $\zeta$ replaced by $\zeta \circ r$, 
				$x$ by $\olsi  x $, $z$ by $\zeta (r(\olsi  x))$ to prove that $\olsi  p = r(\olsi  x)$  satisfies \autoref{Florenzano1982-thm}. It remains to verify Conditions \eqref{eq-in-sphere} and \eqref{HS-ineq} of \autoref{spprtng-lm}. On one hand, for Condition \eqref{eq-in-sphere}, let $p \in S\cap P$. Since $r$ is a retract mapping, it follows $r(p) = p$,  and consequently, $\zeta \circ r (p) = \zeta (p)$. Combining this with condition \eqref{Florenzano1982-condition2} (more precisely condition \eqref{Florenzano1982-condition-mapping}), we obtain 
				$$\langle p, \zeta \circ r (p) \rangle = \langle p, \zeta (p) \rangle  \leq  0,$$ implying that condition $\eqref{eq-in-sphere}$ holds. On the other hand, by Inquality \eqref{HS ineq 1},  Condition \eqref{HS-ineq} holds for $x=\olsi  x$ and $z = \zeta(r(\olsi  x ))$. The proof for this case is over.   
				%By the assumption about $\zeta$ of \autoref{Florenzano1982-thm} (or condition \eqref{Florenzano1982-condition2}), condition \eqref{eq-in-sphere} holds with $\zeta$ replaced by $\zeta \circ r$. By inequality \eqref{ineq1}, inequality 	\eqref{HS-ineq} holds with $x$ replaced by $\olsi  x$, $z$ by $\zeta \circ r (\olsi  x)$. Apply \autoref{spprtng-lm} with $\zeta$ replace by $\zeta \circ r$, $x$ by $\olsi  x$. We thus\footnote{Since $\zeta$ is single-valued, it follows that $P^\circ \cap \zeta \circ r (\olsi  x) \ne \emptyset$ is equivalent to $\zeta \circ r (\olsi  x) \in P^\circ$.} obtain $\zeta \circ r \olsi  (x) \in P^\circ$. 

				%\setcounter{case}{0} 
				%\begin{case} \boldmath {$H \ne span(P)$.  }
				%\end{case}  

				%we obtain $\zeta \big( r( \olsi  x) \big) \in P^0$ with $\olsi  r(\olsi  x) \in S \cap P$. 		
			\end{case}  
			
			\begin{case}{\boldmath \bf{ $P  = span(P)$}} 
%				\newline  
				On one hand, according to \autoref{HS-thm} with $C$ replaced by $\olsi  B \cap P$, we obtain $\olsi  p \in \olsi  B \cap P$ such that 
				\begin{equation*}
					\langle p, \zeta(\olsi  p) \rangle \le \langle \olsi  p, \zeta(\olsi  p) \rangle \;\; \forall p \in \olsi  B \cap B. 
				\end{equation*}
				On the other hand, by the assumption of \autoref{Florenzano1982-thm}, 
				\begin{equation*} 
					\forall p \in S \cap P, \;\; \langle p, \zeta(p) \rangle \le 0. 
				\end{equation*} \hyperref[spprtng-lm2]{Lemma}
				\autoref{spprtng-lm2} implies that $\zeta(\olsi  p) \in P^\circ$. This concludes \autoref{Florenzano1982-thm}. 
			\end{case} 
%		\end{proof} 
%		\qed 
			\subsection{An Alternative Proof of \autoref{Florenzano1982-thm}: when \texorpdfstring{$\zeta$}{\textzeta} is Single-valued and  \bf{$P  = span(P)$} } 
			\label{Florenzano(82)-mapping-proof} 
%			\begin{proof}
				Apply \autoref{HS-thm} to the mapping $\zeta$ on $\olsi  B \cap P$, and obtain some $\olsi  p \in \olsi  B \cap P$ such that
				\begin{equation} \label{ineq1-1}
					\langle p, \zeta(\olsi  p) \rangle  \leq  \langle \olsi  p, \zeta(\olsi  p) \rangle \;\;  \text{ for all } p \in \olsi  B \cap P. 
				\end{equation}
				We want to conclude that $\zeta(\olsi  p) \in P^\circ$. We split the argument into three cases: 
				%	Let $m$ be the dimension of $P$. 
				\begin{itemize}
					\item 
					If $\olsi  p \in int(\olsi  B \cap P)$, since $\zeta(\olsi  p) \in N_{\olsi  B \cap P}(\olsi  p)$, where $N_{\olsi  B \cap P}(\olsi  p)$ is the normal cone to $\olsi  B \cap P$ at $\olsi  p$, we conclude $\zeta(\olsi  p) =0_N$. Consequently,  $\zeta (\olsi  p) \in P^\circ$. Note that this circumstance happens only if the cone  $P$ is $\mathds R^N$.    
					\item If $\olsi  p \in \text{ri}(\olsi  B \cap P)$, i.e., $\olsi  p$ belongs the relative interior\footnote{\label{relative-inteior-et-boundary}For notion of relative interior and relative boundary, see, for example, Section $1.2.2$ in  \cite{FlorenzanoLeVan2001} on page $11$. } of $\olsi  B \cap P$ then $\zeta(\olsi  p) \in N_{\olsi  B \cap P}(\olsi  p)$. We know that since $P$ is the subspace, it follows $N_{\olsi  B \cap P}(\olsi  p) = N_{P}(\olsi  p) = P^\perp = P^\circ$. Hence $\zeta(\olsi  p) \in P^\circ$. 
					%is an interior of $\olsi  B \cap P$ in the space $P$ and from \eqref{ineq1} consequently  $\zeta(\olsi  p)$ is an element in the normal cone to $\olsi  B \cap P$ at $\olsi  p$ in the space $P$, therefore $\zeta (\olsi  p) = 0_N$ and $\zeta(\olsi  p)$ belongs to the polar cone of $P$. 
					\item If $\olsi  p \notin \text{ri}( \olsi  B \cap P)$, then $\olsi  p \in \text{Bd}^\text{r}(\olsi  B \cap P)$, i.e., the relative boundary
					\footnote{See Footnote \ref{relative-inteior-et-boundary}.}
					%\footref{relative-inteior-et-boundary} 
					of $\olsi  B \cap P$.  Since $P$ is the subspace, $\text{Bd}^\text{r}(\olsi  B \cap P) = S \cap P$. Consequently, $\olsi  p \in S \cap P$. By the assumptions about $\zeta$, we deduce that $\langle \olsi  p, \zeta(\olsi  p) \rangle \le 0$. Combining this with Inequality \eqref{ineq1-1}, we obtain, 
					\begin{equation} \label{ineq2-2}
						\langle p, \zeta(\olsi  p) \rangle  \leq 0 \;\;  \text{ for any } p \in \olsi  B \cap P. 
					\end{equation}
					Since $P$ is a cone , it follows that Inequality \eqref{ineq2-2} can be extended to any $p \in P$. As a result,  $\zeta(\olsi  p) \in P^0$. 
				\end{itemize}
				In conclusion, there is some $\olsi  p \in \olsi  B \cap P$ such that $\zeta(\olsi  p) \in P^0$. The proof for the mapping is over.  
		\subsection{Proof of \autoref{spprtng-lm}} 
		\label{prove-supporting-lm}
		%	\begin{proof}[Proof of \autoref{spprtng-lm}] %\nameref{spprtng-lm}
			First we claim that $
			\langle x,z \rangle \leq 0.$ Indeed, the proof of the claim splits into two cases:  
			\setcounter{case}{0}
			\begin{case}{\boldmath{$||x|| = 0$ or $||x||=1$}}. 
				\begin{itemize}
					\item If $||x||=0$, then it is obvious that $\langle x,z \rangle  =0$. 
					\item  
					If $||x||=1$, in this case $x \in S\cap P$. By Condition \eqref{eq-in-sphere}, we consequently obtain  $\langle x,z\rangle  \le 0$.
				\end{itemize} 
				
			\end{case}
			\begin{case}{\bf $ \boldsymbol{0 < || x|| < 1}$}. \\
				Take $p= \frac{x}{||x||}$. Since $P$ is a cone, $x \in P$, and $||p||=1$, it follows $p \in \olsi  B \cap P$. Inequality \eqref{HS-ineq} implies that 
				\begin{equation*}
					\left(1 - 
					\frac{1}{||x||}\right) \langle x, z \rangle   \ge 0. 
				\end{equation*}
				Note that $1 -\frac{1}{||x||}<0$, hence $\langle x, z \rangle  \le 0$. We have finished proving  the claim.  	
			\end{case}
			\noindent We now turn to  show that $z \in P^\circ \cap\zeta(x)$. Since $\langle x,z \rangle \leq 0$, it follows from Inequality \eqref{HS-ineq} that
			\begin{equation} \label{polar-cone-ineq}
				\langle p, z \rangle \leq 0 \;\; \forall p \in \olsi  B \cap P.  
			\end{equation}
			Since $P$ is a cone and $\{x \in P:||x|| \le 1 \} \subset \olsi  B \cap P$, we could extend Inequality \eqref{polar-cone-ineq} to all $p \in P$.\footnote{Indeed, if $p=0$, then $\langle p, z \rangle \leq 0$. If $p\not=0$, define $p'\equiv p/||p||$, then, $p'\in \olsi  B \cap P$. Then, $\langle p', z \rangle \leq 0$ implies $\langle p, z \rangle \leq 0$.} In other words, $z \in P^\circ$. According to Condition \eqref{HS-ineq}, $z \in \zeta(x)$. Therefore, we obtain $z \in P^\circ \cap \zeta(x) $, and this concludes the proof of \autoref{spprtng-lm}. 
\bibliographystyle{spmpsci}
\bibliography{reference3}
%\bibliographystyle{plain}
%\bibliographystyle{spbasic}
%\bibliographystyle{spbasic}

%\bibliographystyle{elsarticle-harv}
%\printbibliography

\end{document}